\pdfoutput=1
\documentclass[11pt,reqno]{amsart}

\usepackage[letterpaper,left=0.80in,right=0.80in,top=0.95in,bottom=0.95in]{geometry}

\usepackage[T1]{fontenc}
\usepackage[utf8]{inputenc}
\usepackage{lmodern}
\usepackage{amsmath,amssymb,amsthm,mathtools}
\usepackage{enumitem}
\usepackage{microtype}
\usepackage{hyperref}
\hypersetup{
  hidelinks,
  pdftitle={On a Class of Hypergeometric Sums via Recurrences and Product Binomial-Harmonic Identities},
  pdfauthor={Narendra Bhandari}
}
\allowdisplaybreaks
\numberwithin{equation}{section}

\newtheorem{theorem}{Theorem}[section]
\newtheorem{proposition}[theorem]{Proposition}
\newtheorem{corollary}[theorem]{Corollary}
\newtheorem{lemma}[theorem]{Lemma}
\theoremstyle{definition}
\newtheorem{example}[theorem]{Example}
\theoremstyle{remark}
\newtheorem{remark}[theorem]{Remark}
\newtheorem*{remark*}{Remark}

\newcommand{\eps}{\varepsilon}
\newcommand{\Bcal}{\mathcal B}
\newcommand{\Hcal}{\mathcal H}
\newcommand{\RePart}{\operatorname{Re}}

\title[Hypergeometric moment recurrences]{Recurrences for Hypergeometric Moments and Binomial-Harmonic Sums}

\author{Narendra Bhandari}
\address{Department of Mathematics, University of North Texas, Denton, TX 76201, USA}
\email{narendra.bhandari@unt.edu}
\email{narenbhandari04@gmail.com}

\date{July 2026}

\begin{document}

\begin{abstract}
Let
\[
F(a,b;c;x)={}_2F_1(a,b;c;x),
\qquad
\Phi_{m,\varepsilon}(\lambda;a,b,c)
=
\int_0^1 x^{m+\lambda}F(a,b;c;\varepsilon x)\,dx,
\quad \varepsilon=\pm1.
\]
We study these moments by deriving and solving a first-order recurrence
in \(m\). This recurrence leads to formulas for higher powers of the
denominator and for denominators of the form \((dn+m+1)^K\), with
applications to product-binomial series and moments of complete elliptic
integrals. Differentiation with respect to \(c\) gives corresponding
recurrences for harmonic-number weights, whose initial values are
described using Bell polynomials, logarithms, zeta values, and
Dirichlet \(L\)-values. Finally, comparison with a terminating
\({}_3F_2(1)\) formula also gives finite hypergeometric and
binomial--harmonic identities.
\end{abstract}
\maketitle

\section{Introduction}\label{sec:introduction}

We study hypergeometric sums containing a denominator of the form
\(n+m+1+\lambda\). Such sums occur in identities with binomial
coefficients, in moments of elliptic integrals, and in series involving
harmonic numbers. Let
\begin{equation}\label{eq:generalFintro}
F(a,b;c;x)
=
{}_2F_1(a,b;c;x)
=
\sum_{n=0}^{\infty}
\frac{(a)_n(b)_n}{(c)_n\,n!}x^n,
\end{equation}
where \((u)_n\) denotes the Pochhammer symbol. For \(m\geq0\),
\(\varepsilon\in\{1,-1\}\), and a complex parameter \(\lambda\), define
\begin{equation}\label{eq:generalPhiIntro}
\Phi_{m,\varepsilon}(\lambda;a,b,c)
=
\int_0^1
x^{m+\lambda}F(a,b;c;\varepsilon x)\,dx.
\end{equation}
Termwise integration gives
\begin{equation}\label{eq:generalSeriesIntro}
\Phi_{m,\varepsilon}(\lambda;a,b,c)
=
\sum_{n=0}^{\infty}
\frac{\varepsilon^n(a)_n(b)_n}{(c)_n\,n!}
\frac{1}{n+m+1+\lambda}.
\end{equation}
We call the sum non-alternating when \(\varepsilon=1\) and alternating
when \(\varepsilon=-1\).

Put \(M=m+\lambda\). The main result is the first-order recurrence
\begin{equation}\label{eq:centralIntroRec}
\varepsilon(M+1-a)(M+1-b)\Phi_{m,\varepsilon}
=
M(M+1-c)\Phi_{m-1,\varepsilon}
+
E_\varepsilon(M;a,b,c),
\end{equation}
where the arguments of the moment functions have been suppressed.
The upper-endpoint term \(E_\varepsilon\) is computed explicitly. In
the non-alternating case, it is a gamma quotient independent of \(M\).
In the alternating case, it is linear in \(M\) and involves only the
value and derivative of \({}_2F_1(a,b;c;z)\) at \(z=-1\). The recurrence
follows from Euler's differential equation by two integrations by parts
and admits an explicit product--sum solution.

The parameter \(\lambda\) allows one recurrence to control every
positive denominator power. Expanding \eqref{eq:generalSeriesIntro} in
powers of \(\lambda\) gives finite formulas for
\[
\sum_{n=0}^{\infty}
\frac{\varepsilon^n(a)_n(b)_n}{(c)_n\,n!}
\frac{1}{(n+m+1)^K},
\qquad K\geq1.
\]
Separating \(m\) into residue classes also treats denominators of the
form \((dn+m+1)^K\), where \(d\) is any positive integer.

Differentiation with respect to \(c\) gives corresponding recurrences
for harmonic-number weights. At an arbitrary base value \(c_0\), the
derivatives produce the shifted harmonic sums
\[
H_n^{(r)}(c_0)
=
\sum_{k=0}^{n-1}\frac{1}{(c_0+k)^r}.
\]
Complete Bell polynomials organize the weights arising from higher
derivatives, and the recurrence at order \(r\) depends only on orders
\(r\) and \(r-1\). At \(c_0=1\), these weights reduce to the usual
harmonic-number combinations. In the non-alternating case, the initial
moments with denominator \(n+1\) admit a Bell-polynomial formula valid
at every order. For rational parameters, these values can be expressed
using logarithms, zeta values, and Dirichlet \(L\)-values.

We also compare the non-alternating moment formula with Kouba's
terminating \({}_3F_2(1)\) identity
\cite[Corollary~3.3]{Kouba2024}. This gives a finite hypergeometric
identity. Differentiating it with respect to the lower parameter gives
finite binomial--harmonic identities. We state a general formula and
give examples with squared central binomial coefficients and with
\(\binom{2n}{n}\binom{4n}{2n}\).

\subsection{The case
\texorpdfstring{\((a,b,c)=(\alpha,1-\alpha,1)\)}
{(a,b,c)=(alpha,1-alpha,1)}}

The formulas become especially simple for
\begin{equation}\label{eq:complementaryIntro}
(a,b,c)=(\alpha,1-\alpha,1),
\qquad 0<\alpha<1.
\end{equation}
Put
\(
a_n(\alpha)
=
\frac{(\alpha)_n(1-\alpha)_n}{(n!)^2}.
\)
Several rational values of \(\alpha\) give products of binomial
coefficients:
\begin{equation}\label{eq:binomialexamples}
\begin{aligned}
a_n\left(\frac12\right)
&=
\frac{1}{16^n}\binom{2n}{n}^{\!2},\\
a_n\left(\frac13\right)
&=
\frac{1}{27^n}\binom{3n}{n}\binom{2n}{n},\\
a_n\left(\frac14\right)
&=
\frac{1}{64^n}\binom{2n}{n}\binom{4n}{2n}.
\end{aligned}
\end{equation}
For this specialization, the non-alternating upper-endpoint term in
\eqref{eq:centralIntroRec} reduces to
\(\sin(\pi\alpha)/\pi\). We obtain formulas for the related moments
with denominator \(2n+m+1\), together with representations in terms
of complete elliptic integrals. The same recurrence also gives
alternating identities, cancellation formulas, and formulas with
higher powers of the denominator.

\subsection{Relation with earlier work}

The main tools used here are classical. Standard references for Gauss
and generalized hypergeometric functions include
\cite{AndrewsAskeyRoy,Bailey,Slater,DLMF}. Related recurrences may also
be studied using integral transforms and computer algebra methods; see
\cite{ParisKaminski,Stanley,Zeilberger,Chyzak,Koutschan}. Parameter
differentiation has been widely used to obtain harmonic-number
identities from hypergeometric formulas; representative references
include
\cite{PauleSchneider,ChuDeDonno,WeiGong,ChuCampbell,Chen2025,Au2026}.
Recent work on related binomial-harmonic families includes
\cite{LiChu2023,LiChu2024,LiChu2025,LiChu2026}.

Product-binomial and elliptic-integral sums have also been studied from
several viewpoints. Central-binomial sums and their special constants
are treated in \cite{BorweinBroadhurstKamnitzer}. Campbell,
D'Aurizio, and Sondow \cite{CampbellDAS} connect hypergeometric
functions, complete elliptic integrals, and Fourier--Legendre
expansions. Campbell and Chen \cite{CampbellChen} obtain explicit
families involving squared central binomial coefficients, while
Cantarini \cite{Cantarini2025} develops related Fourier--Legendre and
Jacobi connection formulas.

For \(\varepsilon=1\), recurrence \eqref{eq:centralIntroRec} is
equivalent, under the substitution \(x=m+\lambda+1\), to the
\({}_3F_2(1)\) recurrence obtained by Chen
\cite[Lemma~2, Eq.~(13)]{Chen2021}. The derivation given here is
different and follows directly from Euler's differential equation and
integration by parts.

The contribution of the present paper is a direct moment-based
derivation of \eqref{eq:centralIntroRec}, the unified computation of its
non-alternating and alternating upper-endpoint terms, and its
product--sum solution. These results give uniform formulas for
denominator powers, denominators of the form \((dn+m+1)^K\),
cancellation identities, and harmonic weights of every order. The
product-binomial applications also give recurrence-based evaluations
of several related families. Comparing these evaluations with Kouba's
terminating formulas \cite{Kouba2024} gives finite hypergeometric and
binomial--harmonic identities.

\subsection{Outline of the paper}

Section~\ref{sec:general} derives and solves the recurrence, including
all positive denominator powers, a first-shift cancellation identity,
and denominators of the form \((dn+m+1)^K\).
Section~\ref{sec:applications} gives product-binomial and
elliptic-integral applications. Sections~\ref{sec:harmonicgeneral} and
\ref{sec:harmonicinitial} treat harmonic-number weights, their initial
moments, and rational values of \(\alpha\). Section~\ref{sec:finiteidentities}
records finite hypergeometric and binomial--harmonic identities.

\section{A recurrence for Gauss hypergeometric moments}
\label{sec:general}

For \(n\geq1\), we use the Pochhammer symbol
\[
(u)_n=u(u+1)\cdots(u+n-1)=\frac{\Gamma(u+n)}{\Gamma(u)},
\qquad
(u)_0=1.
\]
We denote
\begin{equation}\label{eq:generalF}
F(a,b;c;x)={}_2F_1(a,b;c;x),
\end{equation}
where
\(
c\notin\{0,-1,-2,\ldots\}.
\)
For \(\eps\in\{1,-1\}\), put
\(
G_\eps(x)=F(a,b;c;\eps x),
\) and define
\begin{equation}\label{eq:generalPhi}
\Phi_{m,\eps}(\lambda;a,b,c)
=
\int_0^1 x^{m+\lambda}G_\eps(x)\,dx.
\end{equation}

Set
\(
M=m+\lambda.
\)
At \(x=0\), the condition
\(
\RePart M>-1
\)
is sufficient for integrability. In the alternating case,
\(F(a,b;c;-x)\) is analytic on the interval \(0\leq x\leq1\), since
\(-1\) is an ordinary point of the hypergeometric differential
equation. Hence \(\RePart M>-1\) is the only endpoint condition needed
for the alternating integral.

In the non-alternating case, the standard continuation formulas at \(x=1\)
\cite[\S15.8(ii)]{DLMF} show that, when the hypergeometric series does not terminate,
\eqref{eq:generalPhi} converges under
\[
\RePart M>-1,
\qquad
\RePart(c-a-b)>-1.
\]
The case \(c=a+b\), in which \(F(a,b;c;x)\) has
logarithmic growth at \(x=1\), is included. If \(a\) or \(b\) is a
non-positive integer, then \(F(a,b;c;x)\) is a polynomial, and only the
condition \(\RePart M>-1\) is required.

Assume first that neither \(a\) nor \(b\) is a non-positive integer.
The gamma-ratio asymptotic \cite[Eq.~(5.11.13)]{DLMF} gives
\begin{equation}\label{eq:generalCoefficientAsymptotic}
\frac{(a)_n(b)_n}{(c)_n n!}
=
\frac{\Gamma(c)}{\Gamma(a)\Gamma(b)}
n^{a+b-c-1}
\left(1+O\left(\frac1n\right)\right)
\end{equation}
as \(n\to\infty\). In particular,
\[
\left|
\frac{(a)_n(b)_n}{(c)_n n!}
\right|
=
O\left(n^{\RePart(a+b-c)-1}\right).
\]
It follows that, in the non-alternating case, termwise integration is
justified when
\(
\RePart(c-a-b)>-1.
\)
We then obtain
\begin{equation}\label{eq:generalSeries}
\Phi_{m,\eps}(\lambda;a,b,c)
=
\sum_{n=0}^{\infty}
\frac{\eps^n(a)_n(b)_n}{(c)_n n!}
\frac{1}{n+m+1+\lambda}.
\end{equation}

For \(\eps=-1\), the series in \eqref{eq:generalSeries} converges
absolutely when
\(
\RePart(c-a-b)>-1
\)
and converges conditionally when
\[
-2<\RePart(c-a-b)\leq-1.
\]
The latter statement follows from
\eqref{eq:generalCoefficientAsymptotic} and summation by parts (or a
complex-valued form of Dirichlet's test). To identify the series with
the integral throughout the range \(\RePart(c-a-b)>-2\), introduce a
parameter \(0<t<1\). The power series for
\(F(a,b;c;-tx)\) may then be integrated term by term. As
\(t\to1^-\), dominated convergence applies to the integral because the
chosen hypergeometric solution is analytic on a neighborhood of the compact interval \([-1,0]\), while Abel's theorem gives the limit of the
convergent endpoint series. If \(\RePart(c-a-b)\leq-2\), the
alternating integral \eqref{eq:generalPhi} remains well defined whenever
\(\RePart M>-1\), although the series in \eqref{eq:generalSeries} need
not converge.

If \(a\) or \(b\) is a non-positive integer, all the preceding series
terminate, so the identities hold without these convergence
restrictions. In the product-binomial applications considered later,
\[
c-a-b=0,
\]
and both the integral and series representations are therefore valid.

Within the preceding series-convergence ranges,
\eqref{eq:generalSeries} may also be written as
\begin{equation}\label{eq:general3F2}
\Phi_{m,\eps}(\lambda;a,b,c)
=
\frac{1}{m+1+\lambda}
{}_3F_2\left(
\begin{matrix}
a,b,m+1+\lambda\\
c,m+2+\lambda
\end{matrix};
\eps
\right).
\end{equation}
Thus, the recurrence below also gives a relation between nearby \({}_3F_2\) functions.

To evaluate the non-alternating upper-endpoint term, we first record a standard
contiguous identity. It treats the finite, the case $c=a+b$, and
integrably singular cases at \(x=1\) in one formula.

\begin{lemma}\label{lem:contiguousendpoint}
Let
\(
c\notin\{0,-1,-2,\ldots\}\) and \(\RePart(c-a-b)>-1.
\)
Then
\begin{equation}\label{eq:contiguousendpoint}
(1-x)F'(a,b;c;x)
+
(c-a-b)F(a,b;c;x)
=
\frac{(c-a)(c-b)}{c}
F(a,b;c+1;x).
\end{equation}
Consequently,
\begin{equation}\label{eq:generalendpointlimit}
\lim_{x\to1^-}
\left[
(1-x)F'(a,b;c;x)
+
(c-a-b)F(a,b;c;x)
\right]
=
\frac{\Gamma(c)\Gamma(c-a-b+1)}
{\Gamma(c-a)\Gamma(c-b)}.
\end{equation}
\end{lemma}

\begin{proof}
Identity \eqref{eq:contiguousendpoint} is a standard relation between
Gauss hypergeometric functions whose parameters differ by one; see
\cite[\S2.5, Eq.~(2.5.10), p.~97]{AndrewsAskeyRoy}. It also follows by
comparing coefficients in the defining series.
Since
\[
\RePart\bigl((c+1)-a-b\bigr)>0,
\]
Gauss's summation formula \cite[Theorem 2.2.2, p.~66]{AndrewsAskeyRoy} gives
\[
F(a,b;c+1;1)
=
\frac{\Gamma(c+1)\Gamma(c-a-b+1)}
{\Gamma(c-a+1)\Gamma(c-b+1)}.
\]
Letting \(x\to1^-\) in \eqref{eq:contiguousendpoint} and using the gamma property
\(
\Gamma(z+1)=z\Gamma(z)
\)
gives \eqref{eq:generalendpointlimit}.
\end{proof}
\begin{theorem}\label{thm:generalrec}
Let \(m\geq1\), put \(M=m+\lambda\), and assume
\(\RePart M>0\). In the non-alternating case \(\eps=1\), also assume
\(
\RePart(c-a-b)>-1.
\)
Then
\begin{equation}\label{eq:generalrec}
\eps(M+1-a)(M+1-b)\Phi_{m,\eps}
=
M(M+1-c)\Phi_{m-1,\eps}
+
E_\eps(M;a,b,c),
\end{equation}
where \(\Phi_{m,\eps}\) stand for \(\Phi_{m,\eps}(\lambda;a,b,c)\). The upper-endpoint term is
\begin{equation}\label{eq:Eplusgeneral}
E_+(a,b,c)
=
\frac{\Gamma(c)\Gamma(c-a-b+1)}
{\Gamma(c-a)\Gamma(c-b)}
\end{equation}
in the non-alternating case, and
\begin{align}\label{eq:Eminusgeneral}
E_-(M;a,b,c)
&=
(c+a+b-2M-2)
{}_2F_1(a,b;c;-1)-2\left.
\frac{d}{dz}{}_2F_1(a,b;c;z)
\right|_{z=-1}
\end{align}
in the alternating case.
\end{theorem}

\begin{proof}
After replacing \(x\) by \(\eps x\), Euler's differential equation
\cite[\S2.3, Eq.~(2.3.5)]{AndrewsAskeyRoy} becomes
\begin{equation}\label{eq:generalODE}
x(1-\eps x)G_\eps''(x)
+
\{c-\eps(a+b+1)x\}G_\eps'(x)
-
\eps abG_\eps(x)
=
0.
\end{equation}
We first carry out the calculation on a compact interval
\([\delta,r]\), where \(0<\delta<r<1\). This allows both integrations by parts to be carried out in the usual way. The integration-by-parts boundary contributions are combined before the
limits \(\delta\to0^+\) and \(r\to1^-\) are taken; in the
non-alternating case, the separate terms need not have finite limits.

Put
\[
p(x)=x^{M+1}(1-\eps x).
\]
The second-derivative term gives
\[
\int_\delta^r p(x)G_\eps''(x)\,dx
=
\left[p(x)G_\eps'(x)\right]_\delta^r
-
\int_\delta^r p'(x)G_\eps'(x)\,dx.
\]
Since
\[
p'(x)
=
(M+1)x^M-\eps(M+2)x^{M+1},
\]
the remaining coefficient of \(G_\eps'\) is
\begin{align*}
q(x)
&=
x^M\{c-\eps(a+b+1)x\}-p'(x)
\\
&=
(c-M-1)x^M
+
\eps(M+1-a-b)x^{M+1}.
\end{align*}
Integrating this term once more gives
\[
\int_\delta^r q(x)G_\eps'(x)\,dx
=
\left[q(x)G_\eps(x)\right]_\delta^r
-
\int_\delta^r q'(x)G_\eps(x)\,dx,
\]
where
\[
q'(x)
=
M(c-M-1)x^{M-1}
+
\eps(M+1)(M+1-a-b)x^M.
\]

Combining all terms in \eqref{eq:generalODE} on \([\delta,r]\)
gives the corresponding identity with truncated moments and full boundary
contribution
\[
\left[
p(x)G_\eps'(x)+q(x)G_\eps(x)
\right]_\delta^r.
\]
We now let \(\delta\to0^+\) and \(r\to1^-\). The lower-endpoint contribution at \(x=0\) vanishes because \(G_\eps\) and \(G_\eps'\) are bounded near zero
and \(\RePart M>0\). The truncated integrals tend to the moments in
\eqref{eq:generalPhi}. Hence
\begin{align}\label{eq:beforefactor}
0
&=
E_\eps(M;a,b,c)
-
M(c-M-1)\Phi_{m-1,\eps}
\\
&\quad
-
\eps\{(M+1)(M+1-a-b)+ab\}\Phi_{m,\eps},
\notag
\end{align}
where the upper-endpoint term is
\begin{equation}\label{eq:Eupper}
E_\eps(M;a,b,c)
=
\lim_{r\to1^-}
\left[p(r)G_\eps'(r)+q(r)G_\eps(r)\right].
\end{equation}
Since
\[
(M+1)(M+1-a-b)+ab
=
(M+1-a)(M+1-b),
\]
equation \eqref{eq:beforefactor} becomes
\eqref{eq:generalrec} once the upper-endpoint term at \(x=1\) is evaluated.

For \(\eps=1\), we have
\[
q(x)
=
(c-a-b)x^{M+1}
+
(c-M-1)x^M(1-x).
\]
Therefore, the upper-endpoint expression can be written exactly
as
\begin{align*}
p(x)F'(a,b;c;x)+q(x)F(a,b;c;x)
&=
x^{M+1}
\left[
(1-x)F'(a,b;c;x)
+
(c-a-b)F(a,b;c;x)
\right]
\\
&\quad+
(c-M-1)x^M(1-x)F(a,b;c;x).
\end{align*}
The condition
\(
\RePart(c-a-b)>-1
\)
and the standard endpoint behavior of the Gauss function imply that
\[
(1-x)F(a,b;c;x)\longrightarrow0
\qquad (x\to1^-).
\]
Since \(x^M\to1\), Lemma~\ref{lem:contiguousendpoint} gives
\[
E_+(a,b,c)
=
\frac{\Gamma(c)\Gamma(c-a-b+1)}
{\Gamma(c-a)\Gamma(c-b)}.
\]

For \(\eps=-1\), the point \(-1\) is regular. At \(x=1\),
\[
G_-(1)=F(a,b;c;-1),
\qquad
G_-'(1)=-F'(a,b;c;-1).
\]
Moreover,
\(
p(1)=2
\)
and
\(
q(1)=c+a+b-2M-2.
\)
Substitution into \eqref{eq:Eupper} gives
\[
E_-(M;a,b,c)
=
-2F'(a,b;c;-1)
+
(c+a+b-2M-2)F(a,b;c;-1),
\]
which is \eqref{eq:Eminusgeneral}.
\end{proof}

\begin{theorem}
\label{thm:generalsolution}

Let \(m\geq0\) and assume
\(
\RePart\lambda>-1.
\)
In the non-alternating case \(\eps=1\), also assume
\(
\RePart(c-a-b)>-1.
\)
Suppose that
\(
(j+\lambda+1-a)(j+\lambda+1-b)\neq0,
\, 1\leq j\leq m.
\)
Define
\begin{equation}\label{eq:rhoj}
\rho_j(\lambda;a,b,c,\eps)
=
\frac{(j+\lambda)(j+\lambda+1-c)}
{\eps(j+\lambda+1-a)(j+\lambda+1-b)}.
\end{equation}
Then
\begin{align}
\Phi_{m,\eps}(\lambda;a,b,c)
&=
\left(\prod_{\ell=1}^{m}\rho_\ell\right)
\Phi_{0,\eps}(\lambda;a,b,c)
\label{eq:generalsolution}+
\sum_{j=1}^{m}
\frac{E_\eps(j+\lambda;a,b,c)}
{\eps(j+\lambda+1-a)(j+\lambda+1-b)}
\prod_{\ell=j+1}^{m}\rho_\ell.
\end{align}
Empty products are understood to be \(1\), and the sum is empty when
\(m=0\).
\end{theorem}

\begin{proof}
Apply Theorem~\ref{thm:generalrec} with
\(m=j\), then the recurrence becomes
\[
\eps(j+\lambda+1-a)(j+\lambda+1-b)\Phi_{j,\eps}
=
(j+\lambda)(j+\lambda+1-c)\Phi_{j-1,\eps}
+
E_\eps(j+\lambda;a,b,c).
\]
Dividing by
\(
\eps(j+\lambda+1-a)(j+\lambda+1-b)
\)
gives
\begin{equation}\label{eq:firstorderiteration}
\Phi_{j,\eps}
=
\rho_j\Phi_{j-1,\eps}
+
\frac{E_\eps(j+\lambda;a,b,c)}
{\eps(j+\lambda+1-a)(j+\lambda+1-b)}.
\end{equation}

Equation \eqref{eq:firstorderiteration} is a first-order nonhomogeneous recurrence. Iterating it from \(j=1\) to \(j=m\) gives
\[
\begin{aligned}
\Phi_{m,\eps}
&=
\rho_m\rho_{m-1}\cdots\rho_1\Phi_{0,\eps}
+
\sum_{j=1}^{m}
\frac{E_\eps(j+\lambda;a,b,c)}
{\eps(j+\lambda+1-a)(j+\lambda+1-b)}
\rho_m\rho_{m-1}\cdots\rho_{j+1}.
\end{aligned}
\]
This is precisely \eqref{eq:generalsolution}.
\end{proof}
For \(K\ge1\), define
\begin{equation}\label{eq:generalAK}
A_{m,\eps}^{(K)}(a,b,c)
=
\sum_{n=0}^{\infty}
\frac{\eps^n(a)_n(b)_n}{(c)_n n!}
\frac{1}{(n+m+1)^K},
\end{equation}
whenever the series converges.
For \(|\lambda|<m+1\),
\[
\frac{1}{n+m+1+\lambda}
=
\sum_{r=0}^{\infty}
\frac{(-1)^r\lambda^r}{(n+m+1)^{r+1}}.
\]
Whenever the moment series is locally uniformly convergent for
\(\lambda\) near \(0\), the sums may be interchanged, giving
\[
\Phi_{m,\eps}(\lambda;a,b,c)
=
\sum_{r=0}^{\infty}
(-1)^r A_{m,\eps}^{(r+1)}(a,b,c)\lambda^r.
\]
Consequently, for every \(K\ge1\),
\begin{equation}\label{eq:generalpowerextract}
A_{m,\eps}^{(K)}(a,b,c)
=
(-1)^{K-1}[\lambda^{K-1}]
\Phi_{m,\eps}(\lambda;a,b,c)
=
\frac{(-1)^{K-1}}{(K-1)!}
\left.
\frac{\partial^{K-1}}
{\partial\lambda^{K-1}}
\Phi_{m,\eps}(\lambda;a,b,c)
\right|_{\lambda=0}.
\end{equation}
To make this coefficient extraction explicit, we now expand the finite
factors in the product--sum solution of
Theorem~\ref{thm:generalsolution}.
For the explicit coefficient form, define
\begin{align}
\mathcal{R}_m(\lambda)
&=
\eps^m
\frac{(1+\lambda)_m(2-c+\lambda)_m}
     {(2-a+\lambda)_m(2-b+\lambda)_m}
=
\sum_{s=0}^{\infty}r_{m,s}\lambda^s,
\label{eq:Rcoefficients}
\end{align}
with \(\mathcal{R}_0(\lambda)=1\). For \(1\le j\le m\), define
\begin{align}
\mathcal{T}_{m,j}(\lambda)
&=
\eps^{m-j+1}
\frac{(j+1+\lambda)_{m-j}
      (j+2-c+\lambda)_{m-j}}
     {(j+1-a+\lambda)_{m-j+1}
      (j+1-b+\lambda)_{m-j+1}}=
\sum_{s=0}^{\infty}t_{m,j,s}\lambda^s,
\label{eq:Tcoefficients}
\end{align}
and write
\begin{equation}
E_\eps(j+\lambda;a,b,c)
=
\sum_{v=0}^{\infty}e_{\eps,j,v}(a,b,c)\lambda^v.
\label{eq:ecoefficients}
\end{equation}
For simplicity, the parameters \(a,b,c,\varepsilon\) are omitted from the notation \(r_{m,s}\) and \(t_{m,j,s}\). In the non-alternating case,
also assume
\(
\RePart(c-a-b)>-1
\) for the following two theorems.

\begin{theorem}
\label{thm:explicitpowers}
Let \(m\ge0\), \(K\ge1\). Assume that
\(c\notin\{0,-1,-2,\ldots\}\), that
\(
j+1-a\neq0,\, j+1-b\neq0,
\, 1\le j\le m.
\)
Then
\begin{align}
A_{m,\eps}^{(K)}(a,b,c)
&=
\sum_{s=0}^{K-1}
(-1)^s r_{m,s}
A_{0,\eps}^{(K-s)}(a,b,c)+
(-1)^{K-1}
\sum_{j=1}^{m}
\sum_{\substack{u,v\ge0\\u+v=K-1}}
t_{m,j,u}\,
e_{\eps,j,v}(a,b,c).
\label{eq:explicitpowers}
\end{align}
\end{theorem}

\begin{proof}
From \eqref{eq:rhoj}, using
\(\eps^{-q}=\eps^q\) for \(\eps\in\{1,-1\}\), we have
\[
\prod_{\ell=1}^{m}
\rho_\ell(\lambda;a,b,c,\eps)
=
\mathcal{R}_m(\lambda).
\]
Similarly,
\begin{align*}
&\frac{1}
{\eps(j+\lambda+1-a)(j+\lambda+1-b)}
\prod_{\ell=j+1}^{m}
\rho_\ell(\lambda;a,b,c,\eps)
=
\mathcal{T}_{m,j}(\lambda).
\end{align*}
Therefore Theorem~\ref{thm:generalsolution} becomes
\begin{align}
\Phi_{m,\eps}(\lambda;a,b,c)
&=
\mathcal{R}_m(\lambda)
\Phi_{0,\eps}(\lambda;a,b,c)+
\sum_{j=1}^{m}
\mathcal{T}_{m,j}(\lambda)
E_\eps(j+\lambda;a,b,c).
\label{eq:compactgeneralsolution}
\end{align}
And for \(m=0\), we have
\begin{equation}
\Phi_{0,\eps}(\lambda;a,b,c)
=
\sum_{q=0}^{\infty}
(-1)^qA_{0,\eps}^{(q+1)}(a,b,c)\lambda^q.
\label{eq:initialTaylor}
\end{equation}
Indeed, this follows by expanding \((n+1+\lambda)^{-1}\) in powers of
\(\lambda\) and interchanging the two sums.

The Cauchy product of \eqref{eq:Rcoefficients} and
\eqref{eq:initialTaylor} gives
\begin{align*}
\mathcal{R}_m(\lambda)\Phi_{0,\eps}(\lambda)
&=
\sum_{N=0}^{\infty}
\left(
\sum_{s=0}^{N}
r_{m,s}(-1)^{N-s}
A_{0,\eps}^{(N-s+1)}
\right)\lambda^N.
\end{align*}
Likewise,
\begin{align*}
\mathcal{T}_{m,j}(\lambda)
E_\eps(j+\lambda;a,b,c)
&=
\sum_{N=0}^{\infty}
\left(
\sum_{\substack{u,v\ge0\\u+v=N}}
t_{m,j,u}e_{\eps,j,v}
\right)\lambda^N.
\end{align*}
Taking the coefficient of \(\lambda^{K-1}\) in
\eqref{eq:compactgeneralsolution} and using
\[
A_{m,\eps}^{(K)}
=
(-1)^{K-1}[\lambda^{K-1}]\Phi_{m,\eps}(\lambda)
\]
yields \eqref{eq:explicitpowers}, because
\(
(-1)^{K-1}(-1)^{K-1-s}=(-1)^s
\).
\end{proof}

\begin{corollary}
\label{cor:explicitpowercases}
The following
specializations hold.
\begin{enumerate}[label=\textup{(\roman*)}]
\item In the non-alternating case,
\begin{align}
A_{m,1}^{(K)}(a,b,c)
&=
\sum_{s=0}^{K-1}
(-1)^s r_{m,s}
A_{0,1}^{(K-s)}(a,b,c)+
(-1)^{K-1}
\frac{\Gamma(c)\Gamma(c-a-b+1)}
     {\Gamma(c-a)\Gamma(c-b)}
\sum_{j=1}^{m}t_{m,j,K-1}.
\label{eq:explicitpowersplus}
\end{align}

\item In the alternating case, put
\(
f(a,b,c)={}_2F_1(a,b;c;-1),
\,
g(a,b,c)=
\left.
\frac{d}{dz}{}_2F_1(a,b;c;z)
\right|_{z=-1}.
\)
Then
\begin{align}
A_{m,-1}^{(K)}(a,b,c)
&=
\sum_{s=0}^{K-1}
(-1)^s r_{m,s}
A_{0,-1}^{(K-s)}(a,b,c)
\notag\\
&\quad+
(-1)^{K-1}
\sum_{j=1}^{m}
\Big[
t_{m,j,K-1}
\big\{(c+a+b-2j-2)f(a,b,c)-2g(a,b,c)\big\}
\notag\\
&\hspace{39mm}
-2f(a,b,c)t_{m,j,K-2}
\Big],
\label{eq:explicitpowersminus}
\end{align}
where \(t_{m,j,-1}=0\).
\end{enumerate}
\end{corollary}

\begin{proof}
For \(\eps=1\), the upper-endpoint term \eqref{eq:Eplusgeneral} is
independent of \(j+\lambda\). Hence
\[
e_{+,j,0}(a,b,c)
=
\frac{\Gamma(c)\Gamma(c-a-b+1)}
     {\Gamma(c-a)\Gamma(c-b)},
\qquad
e_{+,j,v}(a,b,c)=0\quad(v\ge1),
\]
and \eqref{eq:explicitpowersplus} follows from
\eqref{eq:explicitpowers}.

For \(\eps=-1\), equation \eqref{eq:Eminusgeneral} gives
\begin{align*}
E_-(j+\lambda;a,b,c)
&=
\big\{(c+a+b-2j-2)f(a,b,c)-2g(a,b,c)\big\}
-2f(a,b,c)\lambda.
\end{align*}
Thus only the coefficients of \(\lambda^0\) and \(\lambda^1\) are
nonzero. Substitution into \eqref{eq:explicitpowers} gives
\eqref{eq:explicitpowersminus}.
\end{proof}

\begin{example}\label{ex:noncomplementary}
Let
\[
U_m(\lambda)
=
\sum_{n=0}^{\infty}
\frac{\left(\frac{1}{2}\right)_n\left(\frac{1}{3}\right)_n}{(n!)^2}
\frac{1}{n+m+1+\lambda}.
\]
and put
\(
r_m=\frac{(1)_m^2}{(3/2)_m(5/3)_m},
\) for \(r_0=1.
\)
Then, for every \(m\ge0\),
\begin{equation}\label{eq:noncompsolution}
U_m(\lambda)
=
\frac{\Gamma(7/6)}
{\Gamma(1/2)\Gamma(2/3)}
r_m
\left(
3+\sum_{j=1}^{m}\frac{1}{j^2r_{j-1}}
\right).
\end{equation}
All higher denominator powers follow from
Corollary~\ref{cor:explicitpowercases}\textup{(i)}.
\end{example}

\begin{proof}
Apply Corollary~\ref{cor:explicitpowercases}\textup{(i)} with
\(
(a,b,c,K)=\left(\frac12,\frac13,1,1\right).
\)
The upper-endpoint term is
\(
E=
\frac{\Gamma(7/6)}
{\Gamma(1/2)\Gamma(2/3)}.
\)
From \eqref{eq:Rcoefficients} and \eqref{eq:Tcoefficients},
\[
r_{m,0}
=
\frac{(m!)^2}{(3/2)_m(5/3)_m}
=
r_m, \quad \text{and}\quad 
t_{m,j,0}
=
\frac{(j+1)_{m-j}^2}
{(j+1/2)_{m-j+1}(j+2/3)_{m-j+1}}
=
\frac{r_m}{j^2r_{j-1}}.
\]
Moreover, Gauss summation formula gives
\[
\begin{aligned}
U_0
&=
{}_2F_1\left(\frac12,\frac13;2;1\right)=
\frac{\Gamma(2)\Gamma(7/6)}
{\Gamma(3/2)\Gamma(5/3)}
=
3E.
\end{aligned}
\]
Substitution in Corollary~\ref{cor:explicitpowercases}\textup{(i)}
gives \eqref{eq:noncompsolution}.
\end{proof}
We next use the shift parameter to treat denominators of the form
\((dn+m+1)^K\).

For a complex shift \(\lambda_0\), define
\begin{equation}\label{eq:shiftedinitialpowers}
A_{0,\eps}^{(K)}(\lambda_0;a,b,c)
=
\sum_{n=0}^{\infty}
\frac{\eps^n(a)_n(b)_n}
{(c)_n n!\,(n+1+\lambda_0)^K}.
\end{equation}

Let \(m=d\ell+\rho\), where \(0\le\rho\le d-1\), and put
\begin{equation}\label{eq:lambdarho}
\lambda_\rho=\frac{\rho+1}{d}-1.
\end{equation}
Expand the functions defined in
\eqref{eq:Rcoefficients} and \eqref{eq:Tcoefficients} at
\(\lambda=\lambda_\rho\):
\begin{align}
\mathcal R_\ell(\lambda_\rho+t)
&=
\sum_{s=0}^{\infty}
r_{\ell,s}^{[\lambda_\rho]}t^s,
\label{eq:Rshiftedcoefficients}\\
\mathcal T_{\ell,j}(\lambda_\rho+t)
&=
\sum_{s=0}^{\infty}
t_{\ell,j,s}^{[\lambda_\rho]}t^s,
\qquad 1\le j\le\ell,
\label{eq:Tshiftedcoefficients}
\end{align}
and write
\begin{equation}\label{eq:Eshiftedcoefficients}
E_\eps(j+\lambda_\rho+t;a,b,c)
=
\sum_{v=0}^{\infty}
e_{\eps,j,v}^{[\lambda_\rho]}(a,b,c)t^v.
\end{equation}

\begin{theorem}
\label{thm:coefficient}
Let \(d\ge1\), \(m\ge0\), \(K\ge1\). Write \(m=d\ell+\rho\), where
\(0\le\rho\le d-1\), and let \(\lambda_\rho\) be given by
\eqref{eq:lambdarho}. Assume that
\(
j+\lambda_\rho+1-a\neq0,
\,
j+\lambda_\rho+1-b\neq0,
\, 1\le j\le\ell.
\)
Then
\begin{align}
&\sum_{n=0}^{\infty}
\frac{\eps^n(a)_n(b)_n}
{(c)_n n!\,(dn+m+1)^K}
\notag\\
&\quad=
\frac{1}{d^K}
\sum_{s=0}^{K-1}
(-1)^s
r_{\ell,s}^{[\lambda_\rho]}
A_{0,\eps}^{(K-s)}(\lambda_\rho;a,b,c)+
\frac{(-1)^{K-1}}{d^K}
\sum_{j=1}^{\ell}
\sum_{\substack{u,v\ge0\\u+v=K-1}}
t_{\ell,j,u}^{[\lambda_\rho]}
e_{\eps,j,v}^{[\lambda_\rho]}(a,b,c).
\label{eq:explicitcoefficientformula}
\end{align}
For \(d=1\), one has \(\lambda_\rho=0\), and
\eqref{eq:explicitcoefficientformula} reduces to
Theorem~\ref{thm:explicitpowers}.
\end{theorem}

\begin{proof}
Since
\(
dn+m+1
=
d\bigl(n+\ell+1+\lambda_\rho\bigr),
\)
the series on the left of
\eqref{eq:explicitcoefficientformula} equals
\[
\frac{(-1)^{K-1}}{d^K}
[t^{K-1}]
\Phi_{\ell,\eps}(\lambda_\rho+t;a,b,c).
\]
Substitute \(\lambda=\lambda_\rho+t\) in
Theorem~\ref{thm:generalsolution}. Using
\eqref{eq:Rshiftedcoefficients}--\eqref{eq:Eshiftedcoefficients}
and
\[
\Phi_{0,\eps}(\lambda_\rho+t;a,b,c)
=
\sum_{q=0}^{\infty}
(-1)^q
A_{0,\eps}^{(q+1)}(\lambda_\rho;a,b,c)t^q,
\]
the Cauchy product gives the coefficient of \(t^{K-1}\).
Multiplication by \((-1)^{K-1}/d^K\) yields
\eqref{eq:explicitcoefficientformula}.
\end{proof}

\begin{proposition}
\label{prop:firstshiftcancellation}
Let \(K\ge2\). Then
\begin{align}
&\sum_{n=0}^{\infty}
\frac{(a)_n(b)_n}{(c)_n n!}
\left(
\frac{1}{(n+2)^K}
-
\sum_{s=0}^{K-2}
\frac{(-1)^s r_{1,s}}{(n+1)^{K-s}}
\right)
\notag\\
&\qquad=
(-1)^{K-1}
\left\{
r_{1,K-1}A_{0,1}^{(1)}(a,b,c)
+
\frac{\Gamma(c)\Gamma(c-a-b+1)}
{\Gamma(c-a)\Gamma(c-b)}
t_{1,1,K-1}
\right\}.
\label{eq:firstshiftcancellation}
\end{align}
\end{proposition}

\begin{proof}
Set \(m=1\) in
Corollary~\ref{cor:explicitpowercases}\textup{(i)}. This gives
\begin{align*}
A_{1,1}^{(K)}
&=
\sum_{s=0}^{K-1}
(-1)^s r_{1,s}A_{0,1}^{(K-s)}
+
(-1)^{K-1}
\frac{\Gamma(c)\Gamma(c-a-b+1)}
{\Gamma(c-a)\Gamma(c-b)}
t_{1,1,K-1}.
\end{align*}
Separate the term \(s=K-1\), move the remaining terms to the
left, and use the definition \eqref{eq:generalAK}. This yields
\eqref{eq:firstshiftcancellation}.
\end{proof}
By employing the above Proposition \ref{prop:firstshiftcancellation}, one can easily derive the following identities. 
\begin{example}
The following identities hold:
\begin{align*}
&\sum_{n=0}^{\infty}
\frac1{16^n}\binom{2n}{n}^2
\left(
\frac1{(n+2)^3}
-\frac4{9}\frac1{(n+1)^3}
+\frac8{27}\frac1{(n+1)^2}
\right)
=0,\\
&\sum_{n=0}^{\infty}
\frac1{27^n}\binom{3n}{n}\binom{2n}{n}
\left(
\frac1{(n+2)^3}
-\frac9{20}\frac1{(n+1)^3}
+\frac{117}{400}\frac1{(n+1)^2}
\right)
=-\frac{729\sqrt3}{32000\pi},\\
&\sum_{n=0}^{\infty}
\frac1{432^n}\binom{6n}{3n}\binom{3n}{n}
\left(
\frac1{(n+2)^3}
-\frac{36}{77}\frac1{(n+1)^3}
+\frac{1656}{5929}\frac1{(n+1)^2}
\right)
=-\frac{373248}{2282665\pi}.
\label{eq:add-ordinary-R6-K3}
\end{align*}
\end{example}
 Similarly, taking \(\varepsilon=-1\) yields an analogous cancellation proposition for the alternating case.

\section{Product-binomial and elliptic-integral applications}\label{sec:applications}
We now present several compact closed-form evaluations arising from the
general results established in the previous section, with particular
emphasis on rational arguments. The corresponding generating functions
can also be expressed in terms of complete elliptic integrals. In the
modulus convention,
\[
K(k)=\int_0^{\pi/2}
\frac{d\theta}{\sqrt{1-k^2\sin^2\theta}},
\]
and
\begin{equation}\label{eq:ellipticGF}
\begin{aligned}
{}_2F_1\left(\frac12,\frac12;1;x^2\right)
&=\frac{2}{\pi}K(x),\quad 
{}_2F_1\left(\frac14,\frac34;1;x^2\right)
&=\frac{2}{\pi\sqrt{1+x}}
K\left(\sqrt{\frac{2x}{1+x}}\right),
\qquad 0\le x<1.
\end{aligned}
\end{equation}
The second identity is the quadratic transformation used in
\cite[p.~99]{CampbellDAS}. Consequently, the sums considered below may
be interpreted both as product-binomial series and as moments of
complete elliptic integrals.

The integral representations in this section follow by termwise
integration of the corresponding hypergeometric generating functions.
For the finite evaluations with denominator \(n+m+1\), we specialize
Corollary~\ref{cor:explicitpowercases} to \(K=1\). The families with
denominator \(2n+m+1\) follow from Theorem~\ref{thm:coefficient} with
\(d=2\) and \(K=1\), after separating the even and odd values of \(m\).
Thus, the remaining work consists of evaluating the required initial
values and simplifying the resulting Pochhammer products into binomial
coefficients.

We begin with the product-binomial series that motivated the present work.

\begin{theorem}\label{thm:Rfirst}
For every \(m\ge0\),
\begin{align}\label{eq:R4open}
&\sum_{n=0}^{\infty}
\frac{1}{64^n}\binom{2n}{n}\binom{4n}{2n}
\frac{1}{n+m+1}=
\frac{16\sqrt2\,64^m}
{\pi(2m+1)(4m+3)
\binom{2m}{m}\binom{4m+2}{2m+1}}
\sum_{j=0}^{m}
\frac{1}{64^j}\binom{2j}{j}\binom{4j}{2j}.
\end{align}
Formula \eqref{eq:R4open} gives a recurrence-based answer to the
question raised in \cite[pp.~16--17]{Bhandari2022} and by Stewart
\cite[Eq.~(17)]{Stewart2022}. A different finite evaluation of the same
series was previously obtained by Kouba by taking \(k=1\) in
\cite[Corollary~4.2, Eq.~(4.6)]{Kouba2024}.
\end{theorem}
\begin{proof}
Apply Corollary~\ref{cor:explicitpowercases}\textup{(i)} with
\(
(a,b,c,K)=\left(\frac14,\frac34,1,1\right).
\)
The upper-endpoint term appearing in \eqref{eq:explicitpowersplus} is
\[
E_+\left(\frac14,\frac34,1\right)
=
\frac{1}{\Gamma(1/4)\Gamma(3/4)}
=
\frac{\sqrt2}{2\pi}.
\]
The corresponding initial value is
\(
A_{0,1}^{(1)}\left(\frac14,\frac34,1\right)
=
\Phi_{0,1}\left(0;\frac14,\frac34,1\right)
=
\frac{8\sqrt2}{3\pi}.
\)
The coefficient \(r_{m,0}\) in
Corollary~\ref{cor:explicitpowercases}\textup{(i)} becomes
\[
r_{m,0}
=
\frac{(m!)^2}{(5/4)_m(7/4)_m}
=
\frac{16^m(m!)^2}
{\prod_{r=1}^{m}(4r+1)(4r+3)}
=
\frac{3\cdot2^{6m+1}}{(2m+1)(4m+3)
\binom{2m}{m}\binom{4m+2}{2m+1}}.
\]
Here we use
\(
\prod_{r=1}^{m}(4r+1)(4r+3)
=
\frac{(4m+3)!}{3\,2^{2m+1}(2m+1)!}
\). Also, for \(j\ge1\), the coefficient \(t_{m,j,0}\) satisfies
\[
t_{m,j,0}
=
\frac{16}{3}r_{m,0}
\frac{1}{64^j}\binom{2j}{j}\binom{4j}{2j}.
\]
Therefore, \eqref{eq:explicitpowersplus} gives
\[
A_{m,1}^{(1)}
=
\frac{8\sqrt2}{3\pi}r_{m,0}
\sum_{j=0}^{m}
\frac{1}{64^j}\binom{2j}{j}\binom{4j}{2j}.
\]
Combining these identities yields \eqref{eq:R4open}.
\end{proof}
Comparing \eqref{eq:R4open} with Kouba's alternative evaluation
also gives a finite binomial identity; see
Corollary~\ref{cor:finitequartic} below.

We now turn to the family with denominator \(2n+m+1\), using the
second identity in \eqref{eq:ellipticGF}.

\begin{theorem}\label{thm:83}
For every \(m\ge0\),
\begin{align}\label{eq:83integral}
&\sum_{n=0}^{\infty}
\frac{1}{64^n}\binom{2n}{n}\binom{4n}{2n}
\frac{1}{2n+m+1}
=
\frac{2}{\pi}\int_0^1
\frac{x^m}{\sqrt{1+x}}
K\left(\sqrt{\frac{2x}{1+x}}\right)\,dx.
\end{align}
If \(m=2\ell+1\), the series is one half of the series in
Theorem~\ref{thm:Rfirst} with \(m=\ell\). If \(m=2\ell\), then
\begingroup\footnotesize
\begin{align}\label{eq:83even}
&\sum_{n=0}^{\infty}
\frac{1}{64^n}\binom{2n}{n}\binom{4n}{2n}
\frac{1}{2n+2\ell+1}=
\frac{1}{\pi}
\frac{4^\ell\binom{2\ell}{\ell}}
{(4\ell+1)\binom{4\ell}{2\ell}}
\left(
4\log(1+\sqrt2)
+\sqrt2\sum_{j=1}^{\ell}
\frac{(4j-3)\binom{4j-4}{2j-2}}
{(2j-1)^2 4^{j-1}\binom{2j-2}{j-1}}
\right).
\end{align}
\endgroup
Here \(\ell\) is a non-negative integer.
\end{theorem}
\begin{proof}
By \eqref{eq:binomialexamples}, we have
\[
{}_2F_1(1/4,3/4;1;x^2)
=
\sum_{n=0}^{\infty}
\frac{1}{64^n}\binom{2n}{n}\binom{4n}{2n}x^{2n}.
\]
All terms are non-negative for \(0\le x<1\), and the integrated series converges absolutely. Therefore
\begin{align*}
\sum_{n=0}^{\infty}
\frac{1}{64^n}\binom{2n}{n}\binom{4n}{2n}
\frac{1}{2n+m+1}
&=\int_0^1x^m{}_2F_1(1/4,3/4;1;x^2)\,dx.
\end{align*}
Substitution of \eqref{eq:ellipticGF} proves the integral equality.
For \(m=2\ell+1\), the assertion follows immediately from Theorem~\ref{thm:Rfirst}.

If \(m=2\ell\), apply Theorem~\ref{thm:coefficient} with
\((a,b,c,d,\eps,K)=(1/4,3/4,1,2,1,1)\). Then \(\lambda=-1/2\), and
\begin{equation}\label{eq:evenPhi83}
\sum_{n=0}^{\infty}
\frac{1}{64^n}\binom{2n}{n}\binom{4n}{2n}
\frac{1}{2n+2\ell+1}
=\frac12\Phi_{\ell,1}(-1/2;1/4,3/4,1).
\end{equation}
The required initial value is
\begin{equation}\label{eq:initial83half}
\Phi_{0,1}(-1/2;1/4,3/4,1)
=
\frac{8}{\pi}\log(1+\sqrt2),
\end{equation}
proved in \cite[Example~22]{Bhandari2022}. Since
\(
(3/4)_\ell(5/4)_\ell
=16^{-\ell}\prod_{r=1}^{\ell}(4r-1)(4r+1),
\)
the specialization of Theorem~\ref{thm:coefficient} gives
\begin{align}\label{eq:evenraw83}
\frac12\Phi_{\ell,1}(-1/2;1/4,3/4,1)
&=
\frac12
\frac{(1/2)_\ell^2}{(3/4)_\ell(5/4)_\ell}
\left(
\frac{8}{\pi}\log(1+\sqrt2)
\right.\left.
+
\frac{\sqrt2}{2\pi}
\sum_{j=1}^{\ell}
\frac{(3/4)_{j-1}(5/4)_{j-1}}{(1/2)_j^2}
\right).
\end{align}
The corresponding product identities are
\begin{equation}\label{eq:evenproduct83}
\frac{(1/2)_\ell^2}{(3/4)_\ell(5/4)_\ell}
=
\frac{4^\ell\binom{2\ell}{\ell}}
{(4\ell+1)\binom{4\ell}{2\ell}},
\end{equation}
and
\begin{equation}\label{eq:eventerm83}
\frac{(3/4)_{j-1}(5/4)_{j-1}}{4(1/2)_j^2}
=
\frac{(4j-3)\binom{4j-4}{2j-2}}
{(2j-1)^2 4^{j-1}\binom{2j-2}{j-1}}.
\end{equation}
Indeed,
\(
\prod_{r=1}^{\ell}(4r-1)(4r+1)=(4\ell+1)!!
\) 
and
\(
(1/2)_\ell=\frac{(2\ell)!}{4^\ell\ell!}.
\)
These formulas give \eqref{eq:evenproduct83}; replacing \(\ell\) by \(j-1\) and simplifying gives \eqref{eq:eventerm83}. Substitution in \eqref{eq:evenraw83} proves \eqref{eq:83even}.
\end{proof}

\begin{remark}
The work \cite[p.~99]{CampbellDAS} gives the generating function \eqref{eq:ellipticGF}, records several first values, and explains that the integer and half-integer moments can be evaluated. Theorem~\ref{thm:83} gives the even-shift formula, while the odd shifts follow directly from Theorem~\ref{thm:Rfirst}.
\end{remark}

\begin{theorem}\label{cor:84}
Let
\(
G=\sum_{j=0}^{\infty}\frac{(-1)^j}{(2j+1)^2}
\)
be Catalan's constant. For every \(m\ge0\),
\begin{align}
&\sum_{n=0}^{\infty}
\frac{1}{16^n}\binom{2n}{n}^2
\frac{1}{2n+m+1}
=
\frac{2}{\pi}\int_0^1x^mK(x)\,dx
=
\frac{1}{\pi}
\begin{cases}
\displaystyle
\frac{2\,16^\ell}
{(2\ell+1)^2\binom{2\ell}{\ell}^2}
\sum_{j=0}^{\ell}\frac{\binom{2j}{j}^2}{16^j},
& m=2\ell+1,\\[15pt]
\displaystyle
\frac{\binom{2\ell}{\ell}^2}{16^\ell}
\left(
4G+\frac12\sum_{j=1}^{\ell}
\frac{16^j}{j^2\binom{2j}{j}^2}
\right),
& m=2\ell.
\end{cases}
\label{eq:84}
\end{align}
\end{theorem}

\begin{proof}
By \eqref{eq:binomialexamples} and the first identity in
\eqref{eq:ellipticGF},
\[
{}_2F_1(1/2,1/2;1;x^2)
=
\sum_{n=0}^{\infty}
\frac{1}{16^n}\binom{2n}{n}^2x^{2n}
=
\frac{2}{\pi}K(x).
\]
Termwise integration proves the integral equality in \eqref{eq:84}.
For the finite evaluation, apply Theorem~\ref{thm:coefficient} with \((a,b,c,d,\eps,K)=(1/2,1/2,1,2,1,1)\). Since \(c=1\), equation \eqref{eq:generalrec} remains valid at \(M=0\); together with \eqref{eq:Eplusgeneral}, it gives
\[
\Phi_{0,1}(0;1/2,1/2,1)=\frac4\pi.
\]
The remaining initial value is \(\frac12\Phi_{0,1}(-1/2;1/2,1/2,1)\), which can be evaluated as 
\begin{equation}\label{eq:initial84}
\int_0^1{}_2F_1(1/2,1/2;1;t^2)\,dt=\frac2\pi\int_0^{\pi/2}
\left(\int_0^1\frac{dt}{\sqrt{1-t^2\sin^2\theta}}\right)d\theta =\frac2\pi\int_0^{\pi/2}\frac{\theta}{\sin\theta}\,d\theta=\frac{4G}{\pi},
\end{equation}
see also \cite[p.~818]{Adamchik}. 
For the odd class, the required identities are
\begin{equation}\label{eq:oddproducts84}
\frac{(\ell!)^2}{(3/2)_\ell^2}
=
\frac{16^\ell}{(2\ell+1)^2\binom{2\ell}{\ell}^2},
\qquad
\frac{(3/2)_{j-1}^2}{4(j!)^2}
=
\frac{\binom{2j}{j}^2}{16^j}.
\end{equation}
For the even class, the required identities are
\begin{equation}\label{eq:evenproducts84}
\frac{(1/2)_\ell^2}{(\ell!)^2}
=
\frac{\binom{2\ell}{\ell}^2}{16^\ell},
\qquad
\frac{((j-1)!)^2}{(1/2)_j^2}
=
\frac{16^j}{j^2\binom{2j}{j}^2}.
\end{equation}
Substitution of \eqref{eq:initial84}--\eqref{eq:evenproducts84} in Theorem~\ref{thm:generalsolution} gives \eqref{eq:84}.
\end{proof}

\begin{remark}
Since
\[
\sum_{n=0}^{\infty}
\frac{1}{16^n}\binom{2n}{n}^2
\frac{1}{2n+m+1}
=
\frac12
\sum_{n=0}^{\infty}
\frac{1}{16^n}\binom{2n}{n}^2
\frac{1}{n+(m+1)/2},
\]
the odd and even cases of \eqref{eq:84} recover, respectively, the
positive-integer and positive-half-integer evaluations discussed in
\cite{Adamchik}.
\end{remark}

\begin{lemma}\label{lem:f2g2}
We have
\begin{equation}\label{eq:f2g2}
{}_2F_1(1/2,1/2;1;-1)
=
\frac{1}{\Gamma^2(3/4)}\sqrt{\frac{\pi}{2}},
\end{equation}
and
\begin{equation}\label{eq:g2direct}
\left.\frac{d}{dz}{}_2F_1(1/2,1/2;1;z)\right|_{z=-1}
=
\frac{\Gamma^2(1/4)-4\Gamma^2(3/4)}
{8\sqrt2\,\pi^{3/2}}.
\end{equation}
\end{lemma}

\begin{proof}
Put
\[
F(z)={}_2F_1\left(\frac12,\frac12;1;z\right).
\]
Pfaff's transformation gives
\[
F(z)=(1-z)^{-1/2}F\left(\frac{z}{z-1}\right).
\]
It sends \(z=-1\) to \(1/2\), and Kummer's half-argument value gives
\eqref{eq:f2g2}. Differentiating this identity, using the elliptic
representation of \(F(z)\), Legendre's relation, and
\(K(1/\sqrt2)=\Gamma^2(1/4)/(4\sqrt\pi)\), gives
\eqref{eq:g2direct}. See \cite[\S\S15.8,19.4,19.7,19.8]{DLMF} and
\cite{ChuCampbell}.
\end{proof}

\begin{theorem}\label{cor:85}
For every \(m\ge0\),
\begin{align}
&\sum_{n=0}^{\infty}(-1)^n
\frac{1}{16^n}\binom{2n}{n}^2
\frac{1}{n+m+1}
=
\frac{2}{\pi}\int_0^1
\frac{x^m}{\sqrt{1+x}}
K\left(\sqrt{\frac{x}{1+x}}\right)\,dx
\notag\\
&\quad=
\frac{(-1)^m16^m}
{\sqrt2\,\pi^{3/2}(2m+1)^2\binom{2m}{m}^2}
\sum_{j=0}^{m}(-1)^j
\frac{\binom{2j}{j}^2}{16^j}
\left\{(4j+1)\Gamma^2\left(\frac14\right)
-4\Gamma^2\left(\frac34\right)\right\}.
\label{eq:85}
\end{align}
\end{theorem}
\begin{proof}
Pfaff's transformation and the first identity in
\eqref{eq:ellipticGF} give
\[
{}_2F_1(1/2,1/2;1;-x)
=
\sum_{n=0}^{\infty}(-1)^n
\frac{1}{16^n}\binom{2n}{n}^2x^n
=
\frac{2}{\pi\sqrt{1+x}}
K\left(\sqrt{\frac{x}{1+x}}\right).
\]
Termwise integration proves the integral equality in \eqref{eq:85}.

For the finite evaluation, apply
Corollary~\ref{cor:explicitpowercases}\textup{(ii)} with
\(
(a,b,c,K)=(1/2,1/2,1,1).
\) By Lemma~\ref{lem:f2g2} gives
\[
f\left(\frac12,\frac12,1\right)
=
\frac{1}{\Gamma^2(3/4)}\sqrt{\frac{\pi}{2}}
\quad 
\text{and} \quad
g\left(\frac12,\frac12,1\right)
=
\frac{\Gamma^2(1/4)-4\Gamma^2(3/4)}
{8\sqrt2\,\pi^{3/2}}.
\]
Since \(c=1\), equation \eqref{eq:generalrec} remains valid at \(M=0\).
Together with \eqref{eq:Eminusgeneral}, it gives
\[
A_{0,-1}^{(1)}\left(\frac12,\frac12,1\right)
=
\frac{\Gamma^2(1/4)-4\Gamma^2(3/4)}
{\sqrt2\,\pi^{3/2}}.
\]
From \eqref{eq:Rcoefficients} and \eqref{eq:Tcoefficients}, using
\[
\frac{(-1)^m4^m(m!)^2}{\{(2m+1)!!\}^2}
=
\frac{(-1)^m16^m}
{(2m+1)^2\binom{2m}{m}^2}
\quad \text{and} 
\quad 
\frac{\{(2j-1)!!\}^2}{4^{j-1}(j!)^2}
=
4\frac{\binom{2j}{j}^2}{16^j},
\]
we obtain
\[
r_{m,0}
=
\frac{(-1)^m16^m}
{(2m+1)^2\binom{2m}{m}^2},
\quad 
\text{and}
\quad 
t_{m,j,0}
=
4r_{m,0}(-1)^{j+1}
\frac{\binom{2j}{j}^2}{16^j}.
\]
Substitution of the value of \(r_{m,0}\) proves \eqref{eq:85}.
\end{proof}
The alternating case \(d=2\) formula contains one initial constant that is not simplified here. We therefore state it as a formula in terms of finitely many initial values.

\begin{theorem}\label{cor:86}
For every \(m\ge0\),
\begin{align}\label{eq:86integral}
&\sum_{n=0}^{\infty}(-1)^n
\frac{1}{16^n}\binom{2n}{n}^2
\frac{1}{2n+m+1}
=
\frac{2}{\pi}\int_0^1
\frac{x^m}{\sqrt{1+x^2}}
K\left(\frac{x}{\sqrt{1+x^2}}\right)\,dx.
\end{align}
If \(m=2\ell+1\), the series is one half of the series in
Theorem~\ref{cor:85} with \(m=\ell\). If \(m=2\ell\), then
\begingroup\footnotesize
\begin{align}\label{eq:86}
&\sum_{n=0}^{\infty}(-1)^n
\frac{1}{16^n}\binom{2n}{n}^2
\frac{1}{2n+2\ell+1}
\notag\\
&\quad=
\frac{(-1)^\ell\binom{2\ell}{\ell}^2}{16^\ell}
\left(
\frac12\Phi_{0,-1}\left(\frac{-1}{2};\frac{1}{2},\frac{1}{2},1\right)-
\frac{1}{8\sqrt2\,\pi^{3/2}}
\sum_{j=1}^{\ell}(-1)^{j-1}
\frac{16^j}{j^2\binom{2j}{j}^2}\left\{(4j-1)\Gamma^2\left(\frac14\right)
-4\Gamma^2\left(\frac34\right)\right\}
\right).
\end{align}
\endgroup
\end{theorem}

\begin{proof}
Replacing \(x\) by \(x^2\) in the transformation used above gives
\[
{}_2F_1(1/2,1/2;1;-x^2)
=
\sum_{n=0}^{\infty}(-1)^n
\frac{1}{16^n}\binom{2n}{n}^2x^{2n}
=
\frac{2}{\pi\sqrt{1+x^2}}
K\left(\frac{x}{\sqrt{1+x^2}}\right).
\]
Termwise integration proves \eqref{eq:86integral}; the case \(m=0\)
gives the corresponding initial integral. If \(m=2\ell+1\), then  the assertion follows immediately from Theorem~\ref{cor:85}.
For \(m=2\ell\), apply Theorem~\ref{thm:coefficient} with
\[
(a,b,c,d,\eps,K)=(1/2,1/2,1,2,-1,1).
\]
Use the product identities in \eqref{eq:evenproducts84}, the alternating
upper-endpoint values \eqref{eq:f2g2} and \eqref{eq:g2direct}, and the
initial value
\(\frac12\Phi_{0,-1}(-1/2;1/2,1/2,1)\). Since this initial value is not
further reduced, \eqref{eq:86} is a formula in terms of a single initial
constant, rather than a complete explicit evaluation.
\end{proof}

\section{An all-order lower-parameter recurrence}\label{sec:harmonicgeneral}
The parameter $c$ of the Gauss hypergeometric function produces a natural family of shifted harmonic weights.
For \(s\ge1\), define
\begin{equation}\label{eq:shiftedharmonic}
H_n^{(s)}(c_0)
=
\sum_{k=0}^{n-1}\frac{1}{(c_0+k)^s},
\qquad H_0^{(s)}(c_0)=0.
\end{equation}
At \(c_0=1\), this is the usual generalized harmonic number:
\[
H_n^{(s)}(1)=H_n^{(s)}.
\]
Let \(\Bcal_r\) be the complete exponential Bell polynomial, defined by
the generating function below; see \cite[Ch.~III, \S3.3]{Comtet}.
\begin{equation}\label{eq:BellDef}
\exp\left(\sum_{j\ge1}x_j\frac{t^j}{j!}\right)
=
\sum_{r\ge0}\Bcal_r(x_1,\ldots,x_r)\frac{t^r}{r!},
\qquad \Bcal_0=1.
\end{equation}
Define
\begin{equation}\label{eq:Hcalgeneral}
\Hcal_r(n;c_0)
=
\Bcal_r\bigl(
H_n^{(1)}(c_0),1!H_n^{(2)}(c_0),\ldots,
(r-1)!H_n^{(r)}(c_0)
\bigr).
\end{equation}
We use the shorthand
\begin{equation}\label{eq:Hcal}
\Hcal_r(n)=\Hcal_r(n;1).
\end{equation}
Thus
\begin{align}\label{eq:firstHcal}
\Hcal_0(n)&=1,\quad
\Hcal_1(n)=H_n,\quad 
\Hcal_2(n)=H_n^2+H_n^{(2)},\quad
\Hcal_3(n)=H_n^3+3H_nH_n^{(2)}+2H_n^{(3)}.
\end{align}

\begin{lemma}\label{lem:BellDerivative} For \(n,r\ge0\),
\begin{equation}\label{eq:BellDerivative}
\Hcal_r(n;c_0)
=
(-1)^r(c_0)_n
\left.
\frac{\partial^r}{\partial c^r}\frac{1}{(c)_n}
\right|_{c=c_0}.
\end{equation}
\end{lemma}

\begin{proof}
For \(t\) sufficiently small,
\begin{align*}
\frac{(c_0)_n}{(c_0+t)_n}
&=
\prod_{k=0}^{n-1}
\left(1+\frac{t}{c_0+k}\right)^{-1}=
\exp\left(
\sum_{s=1}^{\infty}
H_n^{(s)}(c_0)\frac{(-t)^s}{s}
\right)=
\sum_{r=0}^{\infty}
\Hcal_r(n;c_0)\frac{(-t)^r}{r!},
\end{align*}
where the last equality follows from \eqref{eq:BellDef} and
\eqref{eq:Hcalgeneral}. On the other hand, Taylor expansion at \(c_0\)
gives
\[
\frac{(c_0)_n}{(c_0+t)_n}
=
(c_0)_n
\sum_{r=0}^{\infty}
\left.
\frac{\partial^r}{\partial c^r}\frac{1}{(c)_n}
\right|_{c=c_0}
\frac{t^r}{r!}.
\]
Comparing coefficients of \(t^r\) proves \eqref{eq:BellDerivative}.
\end{proof}

Now fix \(a,b\) and assume
\begin{equation}\label{eq:harmonicconvergenceassumption}
\RePart(c_0-a-b)>-1.
\end{equation}
Under \eqref{eq:harmonicconvergenceassumption}, the gamma-ratio estimate from Section~\ref{sec:general}, uniformly in a sufficiently small pole-free neighborhood of \(c_0\), bounds every fixed \(c\)-derivative of the summand by \(O\bigl(n^{-1-\eta}(\log(n+1))^r\bigr)\) for some \(\eta>0\). Hence the differentiated series converges locally uniformly, and differentiation with respect to \(c\) may be performed term by term.

Define
\begin{equation}\label{eq:PsiGeneral}
\Psi_{m,r,\eps}(\lambda;a,b,c_0)
=
\sum_{n=0}^{\infty}
\frac{\eps^n(a)_n(b)_n}{(c_0)_n n!}
\frac{\Hcal_r(n;c_0)}{n+m+1+\lambda}.
\end{equation}
By Lemma~\ref{lem:BellDerivative},
\begin{equation}\label{eq:PsiGeneralDerivative}
\Psi_{m,r,\eps}(\lambda;a,b,c_0)
=
(-1)^r
\left.
\frac{\partial^r}{\partial c^r}
\Phi_{m,\eps}(\lambda;a,b,c)
\right|_{c=c_0}.
\end{equation}
For the upper-endpoint derivatives, write
\begin{equation}\label{eq:Egeneralr}
E_{\eps,r}(M;a,b,c_0)
=
(-1)^r
\left.
\frac{\partial^r}{\partial c^r}
E_\eps(M;a,b,c)
\right|_{c=c_0}.
\end{equation}
Let \(\psi(z)=\Gamma'(z)/\Gamma(z)\) be the digamma function, and
let \(\psi^{(k)}\) denote its \(k\)-th derivative. In the
non-alternating case, put
\begin{equation}\label{eq:Deltaj}
\Delta_j(a,b;c_0)
=
\psi^{(j-1)}(c_0)
+\psi^{(j-1)}(c_0-a-b+1)
-\psi^{(j-1)}(c_0-a)
-\psi^{(j-1)}(c_0-b).
\end{equation}
\begin{theorem}\label{thm:generalharmonicrec}
Let \(r\ge0\), \(m\ge1\), and
\(M=m+\lambda\). Assume
\(
\RePart M>0,\,
c_0\notin\{0,-1,-2,\ldots\},\,
\RePart(c_0-a-b)>-1.
\)
Then
\begin{align}\label{eq:generalharmonicrec}
\eps(M+1-a)(M+1-b)\Psi_{m,r,\eps}
&=
M(M+1-c_0)\Psi_{m-1,r,\eps}+rM\Psi_{m-1,r-1,\eps}
+E_{\eps,r}(M;a,b,c_0),
\end{align}
where the term containing \(\Psi_{m-1,-1,\eps}\) is omitted when
\(r=0\), and all \(\Psi\)-terms are evaluated at
\((\lambda;a,b,c_0)\). In the non-alternating case, if
\(
c_0-a-b+1, c_0-a, c_0-b
\notin\{0,-1,-2,\ldots\},
\)
then
\begin{align}\label{eq:Egeneralordinaryr}
E_{+,r}(a,b,c_0)
&=
(-1)^r
\frac{\Gamma(c_0)\Gamma(c_0-a-b+1)}
{\Gamma(c_0-a)\Gamma(c_0-b)}\times
\Bcal_r\bigl(
\Delta_1(a,b;c_0),\ldots,\Delta_r(a,b;c_0)
\bigr).
\end{align}
\end{theorem}

\begin{proof}
Differentiate \eqref{eq:generalrec} \(r\) times with respect to \(c\),
set \(c=c_0\), and multiply by \((-1)^r\). Among the coefficients
multiplying the moments in \eqref{eq:generalrec}, only
\(M(M+1-c)\) depends on \(c\). Leibniz' rule therefore gives
\begin{align*}
&(-1)^r
\left.
\frac{\partial^r}{\partial c^r}
\{M(M+1-c)\Phi_{m-1,\eps}\}
\right|_{c=c_0}=
M(M+1-c_0)\Psi_{m-1,r,\eps}
+rM\Psi_{m-1,r-1,\eps},
\end{align*}
which proves \eqref{eq:generalharmonicrec}.

For \(\eps=1\), let
\[
Q(c)=
\frac{\Gamma(c)\Gamma(c-a-b+1)}
{\Gamma(c-a)\Gamma(c-b)}.
\]
Its \(j\)-th logarithmic derivative at \(c_0\) is
\(\Delta_j(a,b;c_0)\). The Bell-polynomial differentiation formula
\cite[Ch.~III, \S3.3]{Comtet} therefore gives
\[
Q^{(r)}(c_0)
=Q(c_0)
\Bcal_r\bigl(\Delta_1(a,b;c_0),\ldots,\Delta_r(a,b;c_0)\bigr),
\]
and multiplication by \((-1)^r\) proves
\eqref{eq:Egeneralordinaryr}.
\end{proof}

\begin{proposition}\label{prop:generalharmonicsolution}
Let \(r\ge0\), \(m\ge0\), and assume
\(
\RePart\lambda>-1,
c_0\notin\{0,-1,-2,\ldots\},
\RePart(c_0-a-b)>-1.
\)
Suppose that
\[
(j+\lambda+1-a)(j+\lambda+1-b)\neq0,
\qquad 1\le j\le m.
\]
Here \(\rho_j(\lambda;a,b,c_0,\eps)\) is the factor defined in
\eqref{eq:rhoj} with \(c=c_0\). Then
\begin{align}\label{eq:generalharmonicsolution}
\Psi_{m,r,\eps}(\lambda;a,b,c_0)
&=
\left(
\prod_{\ell=1}^{m}
\rho_\ell(\lambda;a,b,c_0,\eps)
\right)
\Psi_{0,r,\eps}(\lambda;a,b,c_0)
\notag\\
&\quad+
\sum_{j=1}^{m}
\frac{
 r(j+\lambda)\Psi_{j-1,r-1,\eps}(\lambda;a,b,c_0)
 +E_{\eps,r}(j+\lambda;a,b,c_0)
}
{\eps(j+\lambda+1-a)(j+\lambda+1-b)}
\prod_{\ell=j+1}^{m}
\rho_\ell(\lambda;a,b,c_0,\eps),
\end{align}
where the term containing \(\Psi_{j-1,r-1,\eps}\) is omitted when
\(r=0\). Empty products are understood to be \(1\), and the sum is
empty when \(m=0\).
\end{proposition}

\begin{proof}
Applying Theorem~\ref{thm:generalharmonicrec} with \(m=j\) and dividing
by \(\eps(j+\lambda+1-a)(j+\lambda+1-b)\) gives
\begin{align*}
\Psi_{j,r,\eps}
&=
\rho_j(\lambda;a,b,c_0,\eps)\Psi_{j-1,r,\eps}+
\frac{
 r(j+\lambda)\Psi_{j-1,r-1,\eps}
 +E_{\eps,r}(j+\lambda;a,b,c_0)
}
{\eps(j+\lambda+1-a)(j+\lambda+1-b)},
\end{align*}
where all \(\Psi\)-terms are evaluated at \((\lambda;a,b,c_0)\).
Iterating this identity from \(j=1\) to \(j=m\) gives
\eqref{eq:generalharmonicsolution}.
\end{proof}
At \(c_0=1\), the shifted harmonic numbers reduce to the ordinary
generalized harmonic numbers. Proposition~\ref{prop:generalharmonicsolution}
reduces the harmonic moments to their initial functions
\(\Psi_{0,r,\eps}(\lambda;a,b,c_0)\). In the next section, we evaluate
these initial values in the non-alternating complementary case
\((a,b,c_0,\eps,\lambda)=(\alpha,1-\alpha,1,1,0)\).

\section{Closed initial harmonic moments}\label{sec:harmonicinitial}
The initial value required in Proposition~\ref{prop:generalharmonicsolution}
is, in general, a function of the shift parameter \(\lambda\). We first
evaluate it at \(\lambda=0\), corresponding to the denominator \(n+1\).
Initial values for higher denominator powers require derivatives with
respect to \(\lambda\) and are not evaluated here. In a neighborhood
of \(c=1\), one has
\begin{equation}\label{eq:Ireduction}
{}_3F_2\left(
\begin{matrix}\alpha,1-\alpha,1\\c,2\end{matrix};1\right)
=
\frac{1}{\alpha(1-\alpha)}
\left[
\frac{\Gamma(c)^2}
{\Gamma(c-\alpha)\Gamma(c-1+\alpha)}
-(c-1)
\right].
\end{equation}
Indeed, integrate the hypergeometric differential equation for
\({}_2F_1(\alpha,1-\alpha;c;x)\) over \([0,1]\), use Gauss's formula,
and continue analytically to \(c=1\).

\begin{theorem}\label{thm:allorderinitial}
Let \(0<\alpha<1\). For every \(r\ge0\),
\begin{align}\label{eq:allorderinitial}
&\sum_{n=0}^{\infty}
\frac{(\alpha)_n(1-\alpha)_n}{(n!)^2}
\frac{\Hcal_r(n)}{n+1}=
\frac{1}{\alpha(1-\alpha)}
\left[
\delta_{r,1}
+(-1)^r\frac{\sin(\pi\alpha)}{\pi}
\Bcal_r\bigl(
\Delta_1(\alpha,1-\alpha;1),\ldots,
\Delta_r(\alpha,1-\alpha;1)
\bigr)
\right].
\end{align}
\end{theorem}

\begin{proof}
Differentiate \eqref{eq:Ireduction} \(r\) times at \(c=1\), use
Lemma~\ref{lem:BellDerivative}, and apply the Bell-polynomial derivative
formula to the gamma quotient.
\end{proof}

At \(\alpha=1/4\), Gauss's rational digamma formula gives
\(\Delta_1(1/4,3/4;1)=6\log2\), and hence
\begin{equation}\label{eq:R4harmonic}
\sum_{n=0}^{\infty}
\frac{1}{64^n}\binom{2n}{n}\binom{4n}{2n}
\frac{H_n}{n+1}
=
\frac{16}{3}-\frac{16\sqrt2\log2}{\pi}.
\end{equation}
At \(\alpha=1/2\),
\[
\Delta_1(1/2,1/2;1)=4\log2,\qquad
\Delta_2(1/2,1/2;1)=-\frac{2\pi^2}{3},\qquad
\Delta_3(1/2,1/2;1)=24\zeta(3),
\]
so
\begin{align}\label{eq:cubicH}
&\sum_{n=0}^{\infty}
\frac{1}{16^n}\binom{2n}{n}^2
\frac{H_n^3+3H_nH_n^{(2)}+2H_n^{(3)}}{n+1}
=
\frac{32\pi^2\log2-256\log^3 2-96\zeta(3)}{\pi}.
\end{align}
The appearance of \(\zeta(3)\) is the first odd-zeta layer in this
family.

\subsection{Rational parameters and Dirichlet
\texorpdfstring{\(L\)}{L}-values}

Let
\(
\alpha=\frac aq,
\, q\ge2,
\,1\le a<q,
\, (a,q)=1.
\)
For a Dirichlet character \(\chi\) modulo \(q\), write
\[
L(s,\chi)=\sum_{n=1}^{\infty}\frac{\chi(n)}{n^s}
\qquad (\RePart s>1),
\]
and let \(\varphi(q)\) denote Euler's totient function. The constants
in Theorem~\ref{thm:allorderinitial} can then be described by standard
formulas for rational polygamma values.

\begin{proposition}\label{prop:Dirichlet}
For \(\alpha=a/q\),
\begin{align}\label{eq:D1cyclo}
\Delta_1(a/q,1-a/q;1)
&=
2\log q-
\sum_{u=1}^{q-1}
\cos\left(\frac{2\pi au}{q}\right)
\log\left(2-2\cos\frac{2\pi u}{q}\right).
\end{align}
For every integer \(j\ge2\),
\begin{align}\label{eq:DjL}
\Delta_j(a/q,1-a/q;1)
&=
(-1)^j(j-1)!
\left[
2\zeta(j)
-
\frac{2q^j}{\varphi(q)}
\sum_{\substack{\chi\ (\mathrm{mod}\ q)\\ \chi(-1)=1}}
\overline{\chi(a)}L(j,\chi)
\right].
\end{align}
The sum is over all even Dirichlet characters modulo \(q\), including
the principal character.
\end{proposition}

\begin{proof}
Gauss rational digamma formula gives \eqref{eq:D1cyclo}. For
\(j\ge2\), the polygamma series and character orthogonality combine the
residue classes \(a\) and \(-a\): odd characters cancel and even
characters double. Substitution in \eqref{eq:Deltaj} gives
\eqref{eq:DjL}.
\end{proof}

\begin{corollary}\label{cor:constantfield}
For rational \(\alpha=a/q\), every initial value in Theorem~\ref{thm:allorderinitial} is an explicit Bell polynomial in logarithms of explicit algebraic numbers, zeta values, and even-character values \(L(j,\chi)\), multiplied by algebraic numbers and powers of \(\pi\).
\end{corollary}

\begin{proof}
Insert \eqref{eq:D1cyclo} and \eqref{eq:DjL} into \eqref{eq:allorderinitial}. Also, \(\sin(\pi a/q)\) is algebraic.
\end{proof}

\subsubsection{Examples}

The following examples show how Proposition~\ref{prop:Dirichlet} is used. Recall that
\[
\Hcal_3(n)=H_n^3+3H_nH_n^{(2)}+2H_n^{(3)}.
\]

\begin{example}[The modulus \(q=3\)]\label{ex:q3}
For modulus three, the principal character is the only even character. Since
\[
L(j,\chi_0)=(1-3^{-j})\zeta(j),
\]
Proposition~\ref{prop:Dirichlet} gives
\begin{equation}\label{eq:Dq3}
\Delta_1(1/3,2/3;1)=3\log 3,
\qquad \Delta_2(1/3,2/3;1)=-\pi^2,
\qquad \Delta_3(1/3,2/3;1)=48\zeta(3).
\end{equation}
Using Theorem~\ref{thm:allorderinitial} and
\(
 a_n(1/3)=27^{-n}\binom{3n}{n}\binom{2n}{n},
\)
we obtain
\begin{align}\label{eq:q3cubicseries}
&\sum_{n=0}^{\infty}
\frac1{27^n}\binom{3n}{n}\binom{2n}{n}
\frac{H_n^3+3H_nH_n^{(2)}+2H_n^{(3)}}{n+1}=-\frac{9\sqrt3}{4\pi}
\left(27\log^3 3-9\pi^2\log 3+48\zeta(3)\right).
\end{align}
The same substitution also gives the first- and second-order harmonic evaluations at \(\alpha=1/3\).
\end{example}

\begin{example}[The modulus \(q=5\)]\label{ex:q5}
Let
\[
\phi=\frac{1+\sqrt5}{2},
\qquad
\chi_5(n)=\left(\frac{n}{5}\right),
\qquad
A=\frac52\log 5+\sqrt5\log \phi.
\]
The even characters modulo five are the principal character and the
quadratic character \(\chi_5\). Gauss's formula and
Proposition~\ref{prop:Dirichlet} give
\begin{align}\label{eq:Dq5}
\Delta_1(1/5,4/5;1)=A,
\quad \Delta_2(1/5,4/5;1)=-\frac{25+6\sqrt5}{15}\pi^2,\quad 
\Delta_3(1/5,4/5;1)=120\zeta(3)+125L(3,\chi_5).
\end{align}
Here we used the standard evaluation for the even quadratic character
modulo five (obtainable from the functional equation and generalized
Bernoulli numbers; see \cite[\S\S12.10--12.12]{Apostol})
\[
L(2,\chi_5)=\frac{4\pi^2}{25\sqrt5} \quad \text{and} \quad
\frac{\sin(\pi/5)/\pi}{(1/5)(4/5)}
=\frac{25\sqrt{10-2\sqrt5}}{16\pi}.
\]
Theorem~\ref{thm:allorderinitial} therefore yields
\begingroup\small
\begin{align}\label{eq:S15}
&\sum_{n=0}^{\infty}
\frac{\left(\frac{1}{5}\right)_n\left(\frac{4}{5}\right)_n}{(n!)^2}
\frac{H_n}{n+1}
=
\frac{25}{4}
\left(
1-\frac{\sqrt{10-2\sqrt5}}{4\pi}A
\right),
\notag\\
&\sum_{n=0}^{\infty}
\frac{\left(\frac{1}{5}\right)_n\left(\frac{4}{5}\right)_n}{(n!)^2}
\frac{H_n^2+H_n^{(2)}}{n+1}
=
\frac{25\sqrt{10-2\sqrt5}}{16\pi}
\left(
A^2-\frac{25+6\sqrt5}{15}\pi^2
\right),
\notag\\
&\sum_{n=0}^{\infty}
\frac{\left(\frac{1}{5}\right)_n\left(\frac{4}{5}\right)_n}{(n!)^2}
\frac{H_n^3+3H_nH_n^{(2)}+2H_n^{(3)}}{n+1}=
-\frac{25\sqrt{10-2\sqrt5}}{16\pi}
\left(
A^3-\frac{25+6\sqrt5}{5}\pi^2A
+120\zeta(3)+125L(3,\chi_5)
\right).
\end{align}
\endgroup
Thus the first two harmonic orders contain only logarithms and powers of \(\pi\), while harmonic order three introduces the value \(L(3,\chi_5)\).
\end{example}

\begin{remark}

Since only even characters occur in \eqref{eq:DjL}, their values at
even integers are algebraic multiples of powers of \(\pi\). At odd
integers, new Dirichlet \(L\)-values may occur, as in
Example~\ref{ex:q5}. For complex characters, conjugate terms combine
to ensure that \(\Delta_j(a/q,1-a/q;1)\) is real.
\end{remark}

\section{Finite hypergeometric and binomial--harmonic identities}
\label{sec:finiteidentities}

We now record some finite identities that follow from the results already
proved. The first one is obtained by comparing the non-alternating moment
recurrence with Kouba's terminating \({}_3F_2(1)\) formula. The harmonic
forms then follow by the lower-parameter differentiation used in
Section~\ref{sec:harmonicgeneral}.

\begin{theorem}\label{thm:finitecomparison}
Let \(m\geq0\), put
\(
\delta=1+a+b-c,
\)
 Then
\begin{align}
\sum_{n=0}^{m}
\frac{(a)_n(b)_n}{(c)_n n!}
&=
\frac{(a+1)_m(b+1)_m}{(c)_m m!}
\sum_{j=0}^{m}
(-1)^j\binom{m}{j}
\frac{(\delta)_j j!}{(a+1)_j(b+1)_j}.
\label{eq:finitecomparison}
\end{align}
\end{theorem}

\begin{proof}
We first assume that
\(\RePart c>0\) and \(\RePart\delta>0\). Put
\(
I_m=\Phi_{m,1}(c-1;c-a,c-b,c).
\)
The recurrence in Theorem~\ref{thm:generalrec} becomes
\begin{equation}\label{eq:finitecomparisonrec}
(m+a)(m+b)I_m
=
m(m+c-1)I_{m-1}
+
\frac{\Gamma(c)\Gamma(\delta)}{\Gamma(a)\Gamma(b)}.
\end{equation}
Gauss summation formula gives
\(
I_0=
\frac{\Gamma(c)\Gamma(\delta)}
{\Gamma(a+1)\Gamma(b+1)}.
\)
Iterating \eqref{eq:finitecomparisonrec}, as in
Theorem~\ref{thm:generalsolution}, gives
\begin{equation}\label{eq:finitecomparisonfirst}
I_m=
\frac{\Gamma(c+m)\Gamma(\delta)m!}
{\Gamma(a+m+1)\Gamma(b+m+1)}
\sum_{n=0}^{m}
\frac{(a)_n(b)_n}{(c)_n n!}.
\end{equation}
On the other hand, \eqref{eq:general3F2} gives
\[
(c+m)I_m
={}_3F_2\left(
\begin{matrix}
c-a,c-b,c+m\\
c,c+m+1
\end{matrix};1
\right).
\]
Applying \cite[Corollary~3.3]{Kouba2024} with
\(
(A,B,C,D)=(c-a,c-b,c,c+1)
\)
and using Gauss formula yields
\begin{equation}\label{eq:finitecomparisonsecond}
I_m=
\frac{\Gamma(c)\Gamma(\delta)}
{\Gamma(a+1)\Gamma(b+1)}
{}_3F_2\left(
\begin{matrix}
-m,1,\delta\\
a+1,b+1
\end{matrix};1
\right).
\end{equation}
Comparing \eqref{eq:finitecomparisonfirst} and
\eqref{eq:finitecomparisonsecond}, and then expanding the terminating
\({}_3F_2\), proves \eqref{eq:finitecomparison}.
\end{proof}

\begin{theorem}\label{cor:finiteharmonic} Let
\(r\geq0\). Then
\begingroup\small
\begin{align}
&\sum_{n=0}^{m}
\frac{(a)_n(b)_n}{(c)_n n!}\Hcal_r(n;c)
\notag\\
&\quad=
\frac{(a+1)_m(b+1)_m}{(c)_m m!}
\sum_{j=0}^{m}
(-1)^j\binom{m}{j}
\frac{(\delta)_j j!}{(a+1)_j(b+1)_j}
\notag\\
&\qquad\quad\times
\Bcal_r\Bigl(
H_m^{(1)}(c)+H_j^{(1)}(\delta),
1!\bigl(H_m^{(2)}(c)-H_j^{(2)}(\delta)\bigr),\ldots,
(r-1)!\bigl(
H_m^{(r)}(c)+(-1)^{r-1}H_j^{(r)}(\delta)
\bigr)\Bigr).
\label{eq:finiteharmonic}
\end{align}
\endgroup
For \(r=0\), the Bell polynomial is understood to be \(\Bcal_0=1\),
and for \(r=1\) the list stops after its first argument. More precisely,
for \(1\leq s\leq r\), the \(s\)-th argument of
\(\Bcal_r\) in \eqref{eq:finiteharmonic} is
\[
(s-1)!\left(
H_m^{(s)}(c)+(-1)^{s-1}H_j^{(s)}(\delta)
\right).
\]
\end{theorem}

\begin{proof}
Apply the operator \((-\partial/\partial c)^r\) to
\eqref{eq:finitecomparison}. Lemma~\ref{lem:BellDerivative} gives the
left-hand side. On the right, the only factor depending on \(c\) in
the \(j\)-th summand is \((\delta)_j/(c)_m\). Its \(s\)-th logarithmic
derivative under \(-\partial/\partial c\) is
\[
(s-1)!\left(
H_m^{(s)}(c)+(-1)^{s-1}H_j^{(s)}(\delta)
\right).
\]
The Bell-polynomial differentiation formula now gives
\eqref{eq:finiteharmonic}.
\end{proof}

\begin{corollary}\label{cor:finitecomplementary}
Let \(0<\alpha<1\), and let \(m,r\) be non-negative integers. Then
\begingroup\small
\begin{align}
&\sum_{n=0}^{m}
\frac{(\alpha)_n(1-\alpha)_n}{(n!)^2}\Hcal_r(n)
\notag\\
&\quad=
(m+\alpha)(m+1-\alpha)
\frac{(\alpha)_m(1-\alpha)_m}{(m!)^2}
\sum_{j=0}^{m}
(-1)^j\binom{m}{j}
\frac{(j!)^2}
{(j+\alpha)(j+1-\alpha)(\alpha)_j(1-\alpha)_j}
\notag\\
&\qquad\quad\times
\Bcal_r\Bigl(
H_m+H_j,
1!\bigl(H_m^{(2)}-H_j^{(2)}\bigr),\ldots,
(r-1)!\bigl(H_m^{(r)}+(-1)^{r-1}H_j^{(r)}\bigr)
\Bigr).
\label{eq:finitecomplementary}
\end{align}
\endgroup
\end{corollary}

\begin{proof}
Set \((a,b,c)=(\alpha,1-\alpha,1)\) in
Theorem~\ref{cor:finiteharmonic} and simplify the Pochhammer factors.
\end{proof}

\begin{corollary}
\label{cor:finitecentral}
For every non-negative integer \(m\),
\begin{align}
\sum_{n=0}^{m}
\frac{1}{16^n}\binom{2n}{n}^{\!2}
&=
\frac{(2m+1)^2}{16^m}\binom{2m}{m}^{\!2}
\sum_{j=0}^{m}
(-1)^j\binom{m}{j}
\frac{16^j}{(2j+1)^2\binom{2j}{j}^{\!2}},
\label{eq:finitecentral0}\\
\sum_{n=0}^{m}
\frac{1}{16^n}\binom{2n}{n}^{\!2}H_n
&=
\frac{(2m+1)^2}{16^m}\binom{2m}{m}^{\!2}
\sum_{j=0}^{m}
(-1)^j\binom{m}{j}
\frac{16^j(H_m+H_j)}
{(2j+1)^2\binom{2j}{j}^{\!2}}.
\label{eq:finitecentral1}
\end{align}
\end{corollary}

\begin{proof}
Take \(\alpha=1/2\) and \(r=0,1\) in
Corollary~\ref{cor:finitecomplementary}, and use
\eqref{eq:binomialexamples}.
\end{proof}

\begin{corollary}[The \(\alpha=1/4\) product-binomial case]
\label{cor:finitequartic}
For every non-negative integer \(m\),
\begingroup\small
\begin{align}
&\sum_{n=0}^{m}
\frac{\binom{2n}{n}\binom{4n}{2n}}{64^n}
\notag\\
&\quad=
\frac{(4m+1)(4m+3)}{64^m}
\binom{2m}{m}\binom{4m}{2m}
\sum_{j=0}^{m}
(-1)^j\binom{m}{j}
\frac{64^j}
{(4j+1)(4j+3)\binom{2j}{j}\binom{4j}{2j}},
\label{eq:finitequartic0}\\
&\sum_{n=0}^{m}
\frac{\binom{2n}{n}\binom{4n}{2n}}{64^n}H_n
\notag\\
&\quad=
\frac{(4m+1)(4m+3)}{64^m}
\binom{2m}{m}\binom{4m}{2m}
\sum_{j=0}^{m}
(-1)^j\binom{m}{j}
\frac{64^j(H_m+H_j)}
{(4j+1)(4j+3)\binom{2j}{j}\binom{4j}{2j}},
\label{eq:finitequartic1}\\
&\sum_{n=0}^{m}
\frac{\binom{2n}{n}\binom{4n}{2n}}{64^n}
\left(H_n^2+H_n^{(2)}\right)
\notag\\
&\quad=
\frac{(4m+1)(4m+3)}{64^m}
\binom{2m}{m}\binom{4m}{2m}
\sum_{j=0}^{m}
(-1)^j\binom{m}{j}
\frac{64^j}
{(4j+1)(4j+3)\binom{2j}{j}\binom{4j}{2j}}
\notag\\
&\qquad\quad\times
\left((H_m+H_j)^2+H_m^{(2)}-H_j^{(2)}\right).
\label{eq:finitequartic2}
\end{align}
\endgroup
\end{corollary}

\begin{proof}
Take \(\alpha=1/4\) and \(r=0,1,2\) in
Corollary~\ref{cor:finitecomplementary}, use
\eqref{eq:binomialexamples}, and note that
\(\Bcal_2(x_1,x_2)=x_1^2+x_2\).
\end{proof}

\section{Conclusion}\label{sec:conclusion}
We derived and solved a first-order recurrence for moments of the Gauss
hypergeometric function in both the ordinary and alternating cases. The
resulting formulas apply to higher powers of the denominator and to
denominators of the form \((dn+m+1)^K\).

For several rational choices of the parameters, the hypergeometric
coefficients reduce to products of binomial coefficients. This leads to
binomial-series and elliptic-integral identities. Differentiation with
respect to \(c\) also gives identities involving harmonic numbers of
every order. For rational parameter values, the required initial values
can be expressed using Bell polynomials, logarithms, zeta values, and
Dirichlet \(L\)-values. A comparison with Kouba's terminating formula
also gives finite hypergeometric and binomial--harmonic identities.

The same method may be useful for other special functions whose
differential equations and endpoint values can be evaluated explicitly.

\section*{Acknowledgments}

The author sincerely thanks Dr.\ John M. Campbell for his encouragement,
valuable comments, and kind support. The author is also grateful to
Dr.\ Omran Kouba for bringing \cite{Chen2021,Kouba2024,Stewart2022}
to his attention.


\small
\begin{thebibliography}{99}

\bibitem{Adamchik}
V. S. Adamchik,
A certain series associated with Catalan's constant,
\emph{Z. Anal. Anwend.} 21 (2002), no.~3, 817--826.

\bibitem{AndrewsAskeyRoy}
G. E. Andrews, R. Askey, and R. Roy,
\emph{Special Functions},
Encyclopedia of Mathematics and its Applications, Vol.~71,
Cambridge University Press, Cambridge, 1999.

\bibitem{Apostol}
T. M. Apostol,
\emph{Introduction to Analytic Number Theory},
Springer, New York, 1976.

\bibitem{Au2026}
K. C. Au,
Discovering hypergeometric series with harmonic numbers via Wilf--Zeilberger seeds,
preprint, arXiv:2602.08721, 2026.

\bibitem{Bailey}
W. N. Bailey,
\emph{Generalized Hypergeometric Series},
Cambridge Tracts in Mathematics and Mathematical Physics, No.~32,
Cambridge University Press, Cambridge, 1935.

\bibitem{Bhandari2022}
N. Bhandari,
Infinite series associated with the ratio and product of central binomial coefficients,
\emph{J. Integer Seq.} 25 (2022), Article 22.6.5.

\bibitem{BorweinBroadhurstKamnitzer}
J. M. Borwein, D. J. Broadhurst, and J. Kamnitzer,
Central binomial sums, multiple Clausen values, and zeta values,
\emph{Experiment. Math.} 10 (2001), no.~1, 25--34.

\bibitem{CampbellChen}
J. M. Campbell and K.-W. Chen,
Explicit identities for infinite families of series involving squared binomial coefficients,
\emph{J. Math. Anal. Appl.} 513 (2022), no.~2, Article 126219.

\bibitem{CampbellDAS}
J. M. Campbell, J. D'Aurizio, and J. Sondow,
On the interplay among hypergeometric functions, complete elliptic integrals, and Fourier--Legendre expansions,
\emph{J. Math. Anal. Appl.} 479 (2019), no.~1, 90--121.

\bibitem{Cantarini2025}
M. Cantarini,
Connection formulas between Fourier--Legendre/Jacobi expansions and central binomial series,
\emph{Integral Transforms Spec. Funct.} 36 (2025), no.~11, 940--960,
\href{https://doi.org/10.1080/10652469.2025.2473061}{doi:10.1080/10652469.2025.2473061}.

\bibitem{Chen2021}
K.-W. Chen,
Clausen's series \({}_3F_2(1)\) with integral parameter differences,
\emph{Symmetry} 13 (2021), no.~10, Article 1783,
\href{https://doi.org/10.3390/sym13101783}{doi:10.3390/sym13101783}.

\bibitem{Chen2025}
K.-W. Chen,
Hypergeometric series and generalized harmonic numbers,
\emph{J. Difference Equ. Appl.} 31 (2025), no.~1, 85--114,
\href{https://doi.org/10.1080/10236198.2024.2388746}{doi:10.1080/10236198.2024.2388746}.

\bibitem{ChuCampbell}
W. Chu and J. M. Campbell,
Harmonic sums from the Kummer theorem,
\emph{J. Math. Anal. Appl.} 501 (2021), Article 125179.

\bibitem{ChuDeDonno}
W. Chu and L. De Donno,
Hypergeometric series and harmonic number identities,
\emph{Adv. in Appl. Math.} 34 (2005), no.~1, 123--137.

\bibitem{Comtet}
L. Comtet,
\emph{Advanced Combinatorics: The Art of Finite and Infinite Expansions},
D. Reidel Publishing Company, Dordrecht, 1974.

\bibitem{Chyzak}
F. Chyzak,
An extension of Zeilberger's fast algorithm to general holonomic functions,
\emph{Discrete Math.} 217 (2000), 115--134.

\bibitem{DLMF}
F. W. J. Olver et al. (eds.),
\emph{NIST Digital Library of Mathematical Functions},
\url{https://dlmf.nist.gov/}.

\bibitem{Kouba2024}
O. Kouba,
Applications of a hypergeometric identity and Ramanujan-like series,
\emph{Bull. Malays. Math. Sci. Soc.} 47 (2024), Article 26,
\href{https://doi.org/10.1007/s40840-023-01627-7}{doi:10.1007/s40840-023-01627-7}.

\bibitem{Koutschan}
C. Koutschan,
Creative telescoping for holonomic functions,
in C. Schneider and J. Bl\"umlein (eds.),
\emph{Computer Algebra in Quantum Field Theory},
Springer, Vienna, 2013, pp.~171--194.

\bibitem{LiChu2023}
C. Li and W. Chu,
Infinite series about harmonic numbers inspired by Ramanujan-like formulae,
\emph{Electron. Res. Arch.} 31 (2023), no.~8, 4611--4636.

\bibitem{LiChu2024}
C. Li and W. Chu,
Gauss' second theorem for \({}_2F_1(1/2)\)-series and novel harmonic series identities,
\emph{Mathematics} 12 (2024), Article 1381.

\bibitem{LiChu2025}
C. Li and W. Chu,
Five classes of binomial/harmonic series of convergence rate \(-1/4\),
\emph{AIMS Math.} 10 (2025), no.~7, 16264--16290.

\bibitem{LiChu2026}
C. Li and W. Chu,
Infinite series of convergence rate \(-1/4\) about generalized harmonic numbers,
\emph{Electron. Res. Arch.} 34 (2026), no.~9, 5907--5940,
\href{https://doi.org/10.3934/era.2026262}{doi:10.3934/era.2026262}.

\bibitem{ParisKaminski}
R. B. Paris and D. Kaminski,
\emph{Asymptotics and Mellin--Barnes Integrals},
Encyclopedia of Mathematics and its Applications, Vol.~85,
Cambridge University Press, Cambridge, 2001.

\bibitem{PauleSchneider}
P. Paule and C. Schneider,
Computer proofs of a new family of harmonic number identities,
\emph{Adv. in Appl. Math.} 31 (2003), no.~2, 359--378.

\bibitem{Slater}
L. J. Slater,
\emph{Generalized Hypergeometric Functions},
Cambridge University Press, Cambridge, 1966.

\bibitem{Stanley}
R. P. Stanley,
Differentiably finite power series,
\emph{European J. Combin.} 1 (1980), no.~2, 175--188.

\bibitem{Stewart2022}
S. M. Stewart,
A simple proof and some applications of an integral representation for
the Catalan numbers,
\emph{Appl. Math. E-Notes} 22 (2022), 637--643.


\bibitem{WeiGong}
C. Wei and D. Gong,
The derivative operator and harmonic number identities,
\emph{Ramanujan J.} 34 (2014), 361--371.

\bibitem{Zeilberger}
D. Zeilberger,
A holonomic systems approach to special functions identities,
\emph{J. Comput. Appl. Math.} 32 (1990), no.~3, 321--368.

\end{thebibliography}
\end{document}